\documentclass[oneside,11pt]{amsart}
\usepackage{amsmath}
\usepackage{amssymb}
\usepackage{graphicx}
\usepackage{caption}
\usepackage{subcaption}
\usepackage{pinlabel}
\usepackage{tikz}
\usepackage{color}

\newtheorem{thm}{Theorem}[section]
\newtheorem{prop}[thm]{Proposition}
\newtheorem{lem}[thm]{Lemma}
\newtheorem{cor}[thm]{Corollary}

\theoremstyle{definition}
\newtheorem{defn}[thm]{Definition}
\newtheorem{rem}[thm]{Remark}

\newtheorem{exmp}[thm]{Example}

\newtheorem{cons}[thm]{Construction}

\newcommand{\abs}[1]{\lvert{#1}\rvert}

\renewcommand{\bar}[1]{\overline{#1}}

\newcommand{\set}[2]{\{\,{#1} \mid {#2} \,\}}

\renewcommand{\emptyset}{\varnothing}

\newcommand{\field}[1]{\mathbb{#1}}

\newcommand{\R}{\field{R}}

\newcommand{\PP}{\field{P}}

\newcommand{\RR}{\field{R}}
\newcommand{\JJ}{\field{J}}
\newcommand{\ZZ}{\field{Z}}
\newcommand{\TT}{\field{T}}
\newcommand{\II}{\field{I}}
\newcommand{\KK}{\field{K}}
\newcommand{\LL}{\field{L}}
\newcommand{\pf}{\mathbf{p/f}}
\newcommand{\p}{\mathbf{p}}
\newcommand{\f}{\mathbf{f}}

\DeclareMathOperator{\CAT}{CAT}

\usepackage{ifthen}

\newcommand{\showcomments}{yes}

\newsavebox{\commentbox}
{\ifthenelse{\equal{\showcomments}{yes}}%
{\footnotemark
    \begin{lrbox}{\commentbox}
    \begin{minipage}[t]{1.25in}\raggedright\sffamily\upshape\tiny
    \footnotemark[\arabic{footnote}]}%
{\begin{lrbox}{\commentbox}}}%
{\ifthenelse{\equal{\showcomments}{yes}}%
{\end{minipage}\end{lrbox}\marginpar{\usebox{\commentbox}}}%
{\end{lrbox}}}

\begin{document}
\title[The coarse geometry of certain right-angled Coxeter groups]{On the coarse geometry of certain right-angled Coxeter groups}

\author{Hoang Thanh Nguyen}
\address{Department of Mathematical Sciences\\
University of Wisconsin-Milwaukee\\
P.O. Box 413\\
Milwaukee, WI 53201\\
USA}
\email{nguyen36@uwm.edu}

\author{Hung Cong Tran}
\address{Department of Mathematics\\
The University of Georgia\\
1023 D. W. Brooks Drive\\
Athens, GA 30605\\
USA}
\email{hung.tran@uga.edu}
\date{\today}

\begin{abstract}
Let $\Gamma$ be a connected, triangle-free, planar graph with at least five vertices that has no separating vertices or edges. If the graph $\Gamma$ is $\mathcal{CFS}$, we prove that the right-angled Coxeter group $G_\Gamma$ is virtually a Seifert manifold group or virtually a graph manifold group and we give a complete quasi-isometry classification of these such groups. Otherwise, we prove that $G_\Gamma$ is hyperbolic relative to a collection of $\mathcal{CFS}$ right-angled Coxeter subgroups of $G_\Gamma$. Consequently, the divergence of $G_\Gamma$ is linear, or quadratic, or exponential. We also generalize right-angled Coxeter groups which are virtually graph manifold groups to certain high dimensional right-angled Coxeter groups (our families exist in every dimension) and study the coarse geometry of this collection. We prove that strongly quasiconvex torsion free infinite index subgroups in certain graph of groups are free and we apply this result to our right-angled Coxeter groups. 
\end{abstract}
\maketitle
\section{Introduction}
For each finite simplicial graph $\Gamma$ the associated \emph{right-angled Coxeter group} $G_\Gamma$ has generating set $S$ equal to the vertices of $\Gamma$, relations $s^2=1$ for each $s$ in $S$ and relations $st = ts$ whenever $s$ and $t$ are adjacent vertices. Graph $\Gamma$ is the \emph{defining graph} of right-angled Coxeter group $G_\Gamma$ and its flag complex $K=K(\Gamma)$ is the \emph{defining nerve} of the group. Therefore, we also denote the right-angled Coxeter group $G_\Gamma$ by $G_K$ where $K$ is the flag complex of $\Gamma$. 

In geometric group theory, groups acting on $\CAT(0)$ cube complexes are fundamental objects and right-angled Coxeter groups provide a rich source of these such groups. The geometry of right-angled Coxeter groups was studied by Caprace \cite{MR2665193, MR3450952}, Davis-Okun \cite{MR1812434}, Dani-Thomas \cite{MR3314816, DATA}, Dani-Stark-Thomas \cite{DST}, Behrstock-Hagen-Sisto \cite{MR3623669}, Levcovitz \cite{IL}, Haulmark-Nguyen-Tran \cite{HNT}, Tran \cite{Tran2017} and others. In this paper, we first study the geometry of right-angled Coxeter groups $G_\Gamma$ whose defining graph $\Gamma$ are connected, triangle-free, planar, has at least 5 vertices, and has no separating vertices or edges (we call them \emph{Standing Assumptions}). Then we generalize a part of work on the such group to certain high dimensional right-angled Coxeter groups.

\subsection{Right-angled Coxeter groups with $\mathcal{CFS}$ defining graphs}

It is well-known from the work of Davis-Januszkiewicz~\cite{MR1783167} that every right-angled Artin group is commensurable (hence, quasi-isometric) to some right-angled
Coxeter group and therefore we are especially interested in right-angled Coxeter groups whose coarse geometry are ``similar'' to the one of a right-angled Artin group. Behrstock-Charney~\cite{MR2874959} prove that the divergence of a one-ended right-angled Artin group is linear or quadratic. Therefore, the divergence of a one-ended right-angled Coxeter which is quasi-isometric to some right-angled Artin group must be linear or quadratic. It has been shown by Dani-Thomas~\cite{MR3314816} and Levcovitz~\cite{IL} that the divergence of a right-angled Coxeter group $G_\Gamma$ is linear or quadratic if and only if $\Gamma$ is $\mathcal{CFS}$ (see Definition~\ref{CFS} for the concept of $\mathcal{CFS}$ graphs). Thus studying right-angled Coxeter groups with $\mathcal{CFS}$ defining graphs is one of the main goal in this paper.

\subsubsection{Quasi-isometric classification of $2$--dimensional right-angled Coxeter groups}

Quasi-isometric classification of groups is one of most essential programs in geometric group theory. A complete solution for quasi-isometric classification of the class of right-angled Coxeter groups is unknown (even in the case of $\mathcal{CFS}$ graphs). 
Behrstock observed that the question on quasi-isometric classification of $\mathcal{CFS}$ right-angled Coxeter groups is appealing but likely difficult (see Question~4.2 \cite{B}). In this paper, we partially answer that question when $\mathcal{CFS}$ defining graphs $\Gamma$ satisfy Standing Assumption. 

The key idea here is that after doing a tree-like decomposition on the graph $\Gamma$ (see Section~\ref{sec:graphdecomposition}), we obtain a tree which we call \emph{visual decomposition tree}. We will give the precise definition of visual decomposition tree later in Section~\ref{sec:graphdecomposition}. Currently, the reader only need to know that each piece of this decomposition is a suspension of distinct points. We observe that the right-angled Coxeter group associated to a piece of this decomposition resembles Seifert fibered space. We then glue these pieces in the pattern of the visual decomposition tree to get a graph manifold where $G_\Gamma$ acts properly and cocompactly. Using the work of Behrstock-Neumann on quasi-isometric classification of graph manifolds, we obtain a quasi-isometric classification theorem for right-angled Coxeter groups with $\mathcal{CFS}$ defining graphs.

\begin{thm}
\label{thm:graphmanifold}
Let $\Gamma$ be a graph satisfying Standing Assumptions. Then:
\begin{enumerate}
\item
\label{item:seifert}
The right-angled Coxeter group $G_{\Gamma}$ is virtually a Seifert manifold group if and only if $\Gamma$ is a suspension of some distinct vertices.
\item
\label{item:graphmanifold}
The right-angled Coxeter group $G_{\Gamma}$ is virtually a graph manifold group if and only if $\Gamma$ is $\mathcal{CFS}$ and it is not a suspension of distinct vertices.
\item
\label{haha} Let $\Gamma$ and $\Gamma'$ be two $\mathcal{CFS}$ graphs satisfying Standing Assumptions. Let $T_r$ and $T_r'$ be two visual decomposition trees of $\Gamma$ and $\Gamma'$ respectively. Then two groups $G_{\Gamma}$ and $G_{\Gamma'}$ are quasi-isometric if and only if $T_r$ and $T_r'$ are bisimilar.
\end{enumerate}
\end{thm}

As we discussed above every right-angled Artin group is quasi-isometric to some $\mathcal{CFS}$ right-angled Coxeter group. A natural question arises is which $\mathcal{CFS}$ right-angled Coxeter groups are quasi-isometric to some right-angled Artin groups. In \cite{B}, Behrstock gives an example of $\mathcal{CFS}$ right-angled Coxeter group which is not quasi-isometric to any right-angled Artin group by using Morse boundary. More precisely, the Morse boundary of the right-angled Coxeter group in his examples contains a circle. Meanwhile, Morse boundaries of all right-angled Artin groups are empty or totally disconnected, this is implicit in \cite{MR3339446} and also follows immediately from Theorem F in \cite{MR3664526}. Therefore, the right-angled Coxeter group in his example is not quasi-isometric to any right-angled Artin group since Morse boundary is a quasi-isometry invariant (see \cite{MR3339446} and also \cite{MC1}). However, it would be natural to conjecture that a one-ended right-angled Coxeter group $G_\Gamma$ is quasi-isometric to some right-angled Artin group if and only if $\Gamma$ is $\mathcal{CFS}$ and the Morse boundary of $G_\Gamma$ is empty or totally disconnected. However, we show that this fact is not true. 

In fact, let $\Gamma$ be a $\mathcal{CFS}$, non-join graph which satisfies Standing Assumptions. By an implicit work in \cite{MR3339446} and the fact that right-angled Coxeter group $G_\Gamma$ can be decomposed as a tree of groups with empty Morse boundary, we observe that $G_\Gamma$ has totally disconnected Morse boundary. However, $G_\Gamma$ is not necessarily quasi-isometric to a right-angled Artin group. More precisely, we give a characterization on defining graph $\Gamma$ for $G_\Gamma$ to be quasi-isometric to a right-angled Artin group. Moreover, we also specify types of right-angled Artin groups which are quasi-isometric to such right-angled Coxeter groups. 

\begin{thm}
\label{thm:racgraag}
Let $\Gamma$ be a $\mathcal{CFS}$, non-join graph satisfying Standing Assumptions and $T_r$ a visual decomposition tree of $\Gamma$. Then the following are equivalent:
\begin{enumerate}
\item The right-angled Coxeter group $G_\Gamma$ is quasi-isometric to a right-angled Artin group.
\item The right-angled Coxeter group $G_\Gamma$ is quasi-isometric to the right-angled Artin group of a tree of diameter at least 3.
\item The right-angled Coxeter group $G_\Gamma$ is quasi-isometric to the right-angled Artin group of a tree of diameter exactly 3.
\item All vertices of the tree $T_r$ are black.
\end{enumerate}
\end{thm}

We remark that a visual decomposition tree of a such graph $\Gamma$ as above is a colored tree whose vertices are colored by black and white and it is constructed in Construction~\ref{cons:keyidea2}. By the above theorem, if the defining graph $\Gamma$ that has a visual decomposition tree $T_r$ containing at least one white vertex (see Example \ref{niceexample}), then the right-angled Coxeter group $G_\Gamma$ is not quasi-isometric to any right-angled Artin group.

\subsubsection{Quasi-isometric classification of high dimensional right-angled Coxter groups} 
As we discuss above, the key tool of the proof of quasi-isometric classification of $\mathcal{CFS}$ right-angled Coxeter groups $G_\Gamma$ with defining graphs satisfying Standing Assumptions (see (\ref{haha}) in Theorem~\ref{thm:graphmanifold}) is to decompose $\Gamma$ into a tree of suspensions of distinct points.
We develop this idea to study right-angled Coxeter groups whose nerves belongs to a collections $\KK_n$ ($n\geq 1$) of certain $n$-dimensional flag complexes which can be decomposed as a tree of simpler flag complexes (see Definition \ref{d2}). We remark that the 1-skeleton of each flag complex in $\KK_n$ is always $\mathcal{CFS}$ and $\KK_1$ is actually the collection of all $\mathcal{CFS}$, non-join graphs satisfying Standing Assumptions.

Each flag complex $K$ in $\KK_n$ (by definition) can be constructed from a $\pf$-bipartite $T$ in a collection $\TT_n$ (see Definitions \ref{d0} and \ref{d2}). The tree $T$ is colored in a way to be described in Section~\ref{s3} and we apply the concept of bisimilarity on such tree $T$ to give a complete quasi-isometric classification of each collection of right-angled Coxeter groups $\{G_K\}_{K\in \KK_n}$.

\begin{thm}
\label{thm:quasihigherdim}
Let $K$ and $K'$ be two flag complexes in $\KK_n$ and we assume that $K$ and $K'$ can be constructed from two trees $T$ and $T'$ in $\TT_n$. Then two right-angled Coxeter groups $G_K$ and $G_{K'}$ are quasi-isometric if and only if two colored trees $T$ and $T'$ are bisimilar after possibly reordering the $\p$-colors by an element of the symmetric group on $2n+2$ elements.
\end{thm}

In \cite{MR2727658}, Behrstock-Januszkiewicz-Neumann study quasi-isometry classification of some high dimensional RAAGs. The nerves of these groups can also be constructed from a tree of certain flag complexes of high dimension. Behrstock-Januszkiewicz-Neumann use the tree structure of the nerves to construct geometric models of the corresponding RAAGs to study the quasi-isometry classification of these such groups. The reader can observe that the strategy of the proof of Theorem~\ref{thm:quasihigherdim} (see Subsection~\ref{s3}) is similar to the one for quasi-isometry classification of RAAGs in \cite{MR2727658}. In fact, we also study quasi-isometry classes of our RACGs by constructing their geometric models. However, the such geometric models are not totally identical to the ones in \cite{MR2727658} and they are actually required certain nontrivial techniques. Moreover, our collection of RACGs is ``richer'' and it ``includes'' the collection of RAAGs in \cite{MR2727658} in term of quasi-isometry classes of both collections (see Theorem~\ref{compare}).



\subsubsection{Strongly quasiconvex subgroups of $\mathcal{CFS}$ right-angled Coxeter groups}

One method to understand the structure of a finitely generated group $G$ is to investigate subgroups of $G$ whose geometry reflects that of $G$. Quasiconvex subgroups of hyperbolic groups is a successful application of this approach. However, quasiconvexity is not as useful for arbitrary finitely generated groups since quasiconvexity depends on a choice of generating set and, in particular, is not preserved under quasi-isometry. In \cite{MR3426695}, Durham-Taylor introduce a strong notion of quasiconvexity in finitely generated groups, called \emph{stability}, which is preserved under quasi-isometry.

Stability agrees with quasiconvexity when ambient groups are hyperbolic. However, a stable subgroup of a finitely generated group is always hyperbolic no matter the ambient group is hyperbolic or not (see \cite{MR3426695}). In some sense, the geometry of a stable subgroup does not reflect completely that of the ambient group. In July 2017, the second author in \cite{T2017} introduces another concept of quasiconvexity, called \emph{strong quasiconvexity}, which is strong enough to be preserved under quasi-isometry and reflexive enough to capture the geometry of ambient groups. This notion was also introduced independently by Genevois \cite{Genevois2017} in September 2017 under the name Morse subgroup. 

There is a strong connection between strong quasiconvexity and stability. More precisely, a subgroup is stable if and only if it is strongly quasiconvex and hyperbolic (see \cite{T2017}). Moreover, these notions agree in hyperbolic setting. Outside hyperbolic setting, there are many strongly quasiconvex subgroups that are not stable.

A natural question arises on which non-hyperbolic group $G$ whose all strongly quasiconvex subgroups of infinite index of $G$ are hyperbolic (i.e. stable). In \cite{T2017}, the second author proves that all strongly quasiconvex subgroups of infinite index of one-ended right-angled Artin groups are stable. In a recent paper (see \cite{Heejoung2017}), Kim proves that all strongly quasiconvex subgroups of infinite index of mapping class group of an oriented, connected, finite type surface with negative Euler characteristic are stable. We prove this fact is true for $G_K$ where $K$ is a flag complex in $\KK_n$

\begin{thm}
\label{good}
Let $K$ be a flag complex in $\KK_n$ and $H$ a strongly quasiconvex subgroup of infinite index of the right-angled Coxeter group $G_K$. Then $H$ is virtually free. In particular, $H$ is stable. 
\end{thm}

We remark here that not all $\mathcal{CFS}$ right-angled Coxeter groups has the property that all infinite index strongly quasiconvex subgroups are all virtually free (or even hyperbolic). We refer the reader to Example~\ref{hayhayqua} for this fact.

The main ingredient for the proof of Theorem \ref{good} is the tree of groups structure of the right-angled Coxeter group $G_\Gamma$ with vertex groups and edge groups satisfying certain conditions. Actually, we prove a stronger result that is applied to such tree of groups in general. More precisely,

\begin{prop}
\label{vuiqua}
Assume a group $G$ is decomposed as a finite graph $T$ of groups that satisfies the following.
\begin{enumerate}
\item For each vertex $v$ of $T$ the vertex group $G_v$ is finitely generated and undistorted. Moreover, any strongly quasiconvex, infinite subgroup of $G_v$ is of finite index.
\item Each edge group is infinite.
\end{enumerate}
Then, if $H$ is a strongly quasiconvex, torsion free subgroup of $G$ of infinite index, then $H$ is a free subgroup.
\end{prop}

\subsection{Right-angled Coxeter groups with arbitrary defining graphs satisfying Standing Assumptions}

In general case (when the graph $\Gamma$ is not necessarily $\mathcal{CFS}$), we prove that if $\Gamma$ satisfies Standing Assumptions, the associated right-angled Coxeter group $G_\Gamma$ is hyperbolic relative to a certain collection of $\mathcal{CFS}$ right-angled Coxeter subgroups. More precisely, 

\begin{thm}
\label{sosonice}
Let $\Gamma$ be a graph satisfying Standing Assumptions. There is a collection $\JJ$ of $\mathcal{CFS}$ subgraph of $\Gamma$ such that the right-angled Coxeter group $G_\Gamma$ is relatively hyperbolic with respect to the collection $\PP=\set{G_J}{J\in \JJ}$. 
\end{thm}

For the proof of Theorem~\ref{sosonice} we carefully investigate the tree structure of the defining graph and use results in \cite[Theorem A']{MR2665193, MR3450952} and \cite[Corollary 1.14]{MR2153979} to figure out the relatively hyperbolic structure of group $G_\Gamma$. The investigation of the such tree structure for proof of Theorem~\ref{sosonice} is quite technical and we refer the reader to Section~\ref{nice3} for the details.

By exploring the relatively hyperbolic structure of groups in Theorem~\ref{sosonice} we can take an advantage on Theorem~\ref{thm:graphmanifold} to study quasi-isometry classification of right-angled Coxeter groups even in the case of non-$\mathcal{CFS}$ defining graphs. In fact by Theorem \ref{sosonice} these such groups are relatively hyperbolic with respect to collections of $\mathcal{CFS}$ right-angled Coxeter groups. Therefore, if we know the difference in term of quasi-isometry between two such peripheral structures of two relatively hyperbolic groups $G_\Omega$ and $G_{\Omega'}$ by Theorem \ref{thm:graphmanifold}, we can distinguish $G_\Omega$ and $G_{\Omega'}$ also in term of quasi-isometry. We refer the reader to Example~\ref{newnew} for this application.


Theorem~\ref{sosonice} also contributes to study the divergence of right-angled Coxeter groups. Behrstock-Hagen-Sisto in \cite{MR3623669} show that the divergence of a one-ended right-angled Coxeter group is either exponential or bounded above by a polynomial. Dani-Thomas in \cite{MR3314816} also show that for every positive integer $d$, there is a right-angled Coxeter group with divergence $x^d$. However by combining Theorem \ref{sosonice} with results in \cite[Theorem 1.1]{MR3314816} and \cite[Theorem 1.3]{Sisto}, the divergence functions of one-ended right-angled Coxeter groups $G_\Gamma$ of planar, triangle-free graphs $\Gamma$ are quite simple. More precisely,
\begin{cor}
\label{cor:divofracg}
Let $\Gamma$ be a graph satisfying Standing Assumptions. Then the divergence of the right-angled Coxeter group $G_\Gamma$ is linear, or quadratic, or exponential.
\end{cor}

\subsection{Overview}
In Section~\ref{sec:preliminaries} we review some concepts in geometric group
theory and $3$--manifold theory. In Section~\ref{sec:graphdecomposition} we study the ``tree structure'' of graphs satisfying Standing Assumption. In Section~\ref{sec:racgandgraphmanifold}, we study right-angled Coxeter groups with planar defining graph. We give the proof of Theorem~\ref{thm:graphmanifold} and Theorem~\ref{thm:racgraag} in Section~\ref{nice2}. The proof of Theorem~\ref{sosonice} is given in Section~\ref{nice3}. In Section~\ref{sec:generalization} we generalizes Theorem~\ref{thm:graphmanifold} to a certain high dimensional right-angled Coxeter groups. We give the proof of Theorem~\ref{thm:quasihigherdim} in Section~\ref{s3}. In Section~\ref{sqc}, we study strongly quasiconvex subgroups of $\mathcal{CFS}$ right-angled Coxeter groups. We give proofs of Theorem~\ref{good} and Proposition~\ref{vuiqua} in Section~\ref{hihihi}.

\subsection{Acknowledgments}
The authors would like to thank Chris Hruska and Jason Behrstock for their very helpful conversations and suggestions. The authors are grateful for the insightful comments of the referees that have helped improve the exposition of this paper. We especially appreciate the referee's suggestions on Lemma~\ref{hihi} and Proposition~\ref{hayqua} that lead to a stronger version of Proposition~\ref{vuiqua}.

\section{Preliminaries}
\label{sec:preliminaries}

In this section, we review some concepts in geometric group theory and 3-manifold theory: right-angled Coxeter groups, Davis complexes, right-angled Artin groups, relatively hyperbolic groups, graph manifolds, and mixed manifold. We discuss the work of Caprace \cite{MR2665193, MR3450952}, Behrstock-Hagen-Sisto \cite{MR3623669}, and Dani-Thomas \cite{MR3314816} on peripheral structures of relatively hyperbolic right-angled Coxeter groups and divergence of right-angled Coxeter groups. We also discuss the work of Gersten \cite{Gersten94} and Kapovich--Leeb \cite{KL98} on divergence of $3$--manifold groups. We also mention the concept of colored graphs and the bisimilarity equivalence relation on these such graphs. Lastly, we review the works of Behrstock-Neumann \cite{MR2376814} and Gordon \cite{MR2199348} on connections between right-angled Artin groups and 3--manifold groups.

\subsection{Right-angled Coxeter groups and their relatively hyperbolic structures}

We first review the concepts of right-angled Coxeter groups and Davis complexes.

\begin{defn}
Given a finite simplicial graph $\Gamma$, the associated \emph{right-angled Coxeter group} $G_\Gamma$ is generated by the set S of vertices of $\Gamma$ and has relations $s^2 = 1$ for all $s$ in $S$ and $st = ts$ whenever $s$ and $t$ are adjacent vertices. Graph $\Gamma$ is the \emph{defining graph} of right-angled Coxeter group $G_\Gamma$ and its flag complex $K=K(\Gamma)$ is the \emph{defining nerve} of the group. Sometimes, we also denote the right-angled Coxeter group $G_\Gamma$ by $G_K$ where $K$ is the flag complex of $\Gamma$. 

Let $S_1$ be a subset of $S$. The subgroup of $G_\Gamma$ generated by $S_1$ is a right-angled Coxeter group $G_{\Gamma_1}$, where $\Gamma_1$ is the induced subgraph of $\Gamma$ with vertex set $S_1$ (i.e. $\Gamma_1$ is the union of all edges of $\Gamma$ with both endpoints in $S_1$). The subgroup $G_{\Gamma_1}$ is called a \emph{special subgroup} of $G_\Gamma$.
\end{defn}

\begin{defn}
Given a finite simplicial graph $\Gamma$, the associated \emph{Davis complex} $\Sigma_\Gamma$ is a cube complex constructed as follows. For every $k$--clique, $T \subset \Gamma$, the special subgroup $G_T$ is isomorphic to the direct product of $k$ copies of $Z_2$. Hence, the Cayley graph of $G_T$ is isomorphic to the 1--skeleton of a $k$--cube. The Davis complex $\Sigma_\Gamma$ has 1--skeleton the Cayley graph of $G_\Gamma$, where edges are given unit length. Additionally, for each $k$--clique, $T \subset \Gamma$, and coset $gG_T$, we glue a unit $k$--cube to $gG_T \subset\Sigma_\Gamma$. The Davis complex $\Sigma_\Gamma$ is a $\CAT(0)$ space and the group $G_\Gamma$ acts properly and cocompactly on the Davis complex $\Sigma_\Gamma$ (see \cite{MR2360474}).
\end{defn}

We now review the concept of relatively hyperbolic groups.

\begin{defn}
Given a finitely generated group $G$ with Cayley graph $\Gamma(G,S)$ equipped with the path metric and a finite collection $\PP$ of subgroups of G, one can construct the \emph{coned off Cayley graph} $\hat{\Gamma}(G,S,\PP)$ as follows: For each left coset $gP$ where $P\in \PP$, add a vertex $v_{gP}$, called a \emph{peripheral vertex}, to the Cayley graph $\Gamma(G,S)$ and for each element $x$ of $gP$, add an edge $e(x,gP)$ of length 1/2 from $x$ to the vertex $v_{gP}$. This results in a metric space that may not be proper (i.e. closed balls need not be compact).
\end{defn}

\begin{defn} [Relatively hyperbolic group]
\label{rel}
A finitely generated group $G$ is \emph{hyperbolic relative to a finite collection $\PP$ of subgroups of $G$} if the coned off Cayley graph is $\delta$--hyperbolic and \emph{fine} (i.e. for each positive number $n$, each edge of the coned off Cayley graph is contained in only finitely many circuits of length $n$).
Each group $P\in \PP$ is a \emph{peripheral subgroup} and its left cosets are \emph{peripheral left cosets} and we denote the collection of all peripheral left cosets by $\Pi$.

\end{defn}

\begin{thm}\cite[Corollary 1.14]{MR2153979}
\label{dru}
If a group $G$ is hyperbolic relative to $\{H_1,\cdots, H_m\}$, and each $H_i$ is hyperbolic relative to a collection of subgroups $\{H_i^1,H_i^2,\cdots,H_i^{n_i}\}$ then $G$ is hyperbolic relative to the collection
\[
\set{H_i^j}{i\in\{1,2,\cdots,m\},j\in\{1,2,\cdots,n_i\}}.
\]
\end{thm}

In the rest of this subsection, we discuss the work of Caprace \cite{MR2665193, MR3450952} and Behrstock-Hagen-Sisto \cite{MR3623669} on peripheral structures of relatively hyperbolic right-angled Coxeter groups.

\begin{thm}[Theorem A' in \cite{MR2665193, MR3450952}]
\label{th1}
Let $\Gamma$ be a simplicial graph and $\JJ$ be a collection of induced subgraphs of $\Gamma$. Then the right-angled Coxeter groups $G_\Gamma$ is hyperbolic relative to the collection $\PP=\set{G_J}{J\in\JJ}$ if and only if the following three conditions hold:
\begin{enumerate}
\item If $\sigma$ is an induced 4-cycle of $\Gamma$, then $\sigma$ is an induced 4-cycle of some $J\in \JJ$.
\item For all $J_1$, $J_2$ in $\JJ$ with $J_1\neq J_2$, the intersection $J_1\cap J_2$ is empty or $J_1\cap J_2$ is a complete subgraph of $\Gamma$.
\item If a vertex $s$ commutes with two non-adjacent vertices of some $J$ in $\JJ$, then $s$ lies in $J$.
\end{enumerate}
\end{thm}

\begin{thm}[Theorem B in \cite{MR2665193, MR3450952}]
\label{C1}
Let $\Gamma$ be a simplicial graph. If $G_\Gamma$ is relatively hyperbolic with respect to finitely generated subgroups $H_1, \cdots, H_m$, then each $H_i$ is conjugate to a special subgroup of $G_\Gamma$.
\end{thm}

\begin{thm}[Theorem I in \cite{MR3623669}]
\label{n1}
Let $\mathcal{T}$ be the class consisting of the finite simplicial graphs $\Lambda$ such that $G_\Lambda$ is strongly algebraically thick. Then for any finite simplicial graph $\Gamma$ either: $\Gamma \in \mathcal{T}$, or there exists a collection $\JJ$ of induced subgraphs of $\Gamma$ such that $\JJ \subset \mathcal{T}$ and $G_\Gamma$ is hyperbolic relative to the collection $\PP=\set{G_J}{J \in \JJ}$ and this peripheral structure is minimal.
\end{thm}

\begin{rem}
In Theorem \ref{n1} we use the notion of \emph{strong algebraic thickness} which is introduced in \cite{MR3421592} and is a sufficient condition for a group to be non-hyperbolic relative to any collection of proper subgroups. We refer the reader to \cite{MR3421592} for more details. The following theorem from \cite{MR3623669} characterizes all strongly algebraically thick right-angled Coxeter groups and it will prove useful for studying peripheral subgroups of relatively hyperbolic right-angled Coxeter groups. 
\end{rem}

\begin{thm}[Theorem II in \cite{MR3623669}]
\label{n2}
Let $\mathcal{T}$ be the class of finite simplicial graphs whose corresponding right-angled Coxeter groups are strongly algebraically thick. Then $\mathcal{T}$ is the smallest class of graphs satisfying the following conditions:
\begin{enumerate}
\item The 4-cycle lies in $\mathcal{T}$.
\item Let $\Gamma \in \mathcal{T}$ and let $\Lambda \subset \Gamma$ be an induced subgraph which is not a complete graph. Then the graph obtained from $\Gamma$ by coning off $\Lambda$ is in $\mathcal{T}$.
\item Let $\Gamma_1,\Gamma_2 \in \mathcal{T}$ and suppose there exists a graph $\Gamma$, which is not a complete graph, and which arises as a subgraph of each of the $\Gamma_i$. Then the union $\Lambda$ of $\Gamma_1$, $\Gamma_2$ along $\Gamma$ is in $\mathcal{T}$, and so is any graph obtained from $\Lambda$ by adding any collection of edges joining vertices in $\Gamma_1-\Gamma$ to vertices of $\Gamma_2-\Gamma$.
\end{enumerate}
\end{thm}

\subsection{Divergence of right-angled Coxeter groups and $3$--manifold groups} Roughly speaking, divergence is a quasi-isometry invariant that measures the circumference of a ball of radius $n$ as a function of $n$. We refer the reader to \cite{MR1254309} for a precise definition. In this section, we state some theorems about divergence of certain right-angled Coxeter groups and $3$-manifold groups which will be used later in this paper.

\subsubsection{Divergence of right-angled Coxeter groups}
\begin{thm}[\cite{MR3623669}]
\label{bh}
The divergence of a right-angled Coxeter group is either exponential (if the group is relatively hyperbolic) or bounded above by a polynomial (if the group is strongly algebraically thick).
\end{thm}

\begin{defn}
\label{CFS}
Given a graph $\Gamma$, define the associated \emph{four-cycle} graph $\Gamma^4$ as follows. The vertices of $\Gamma^4$ are the induced loops of length four (i.e. four-cycles) in $\Gamma$. Two vertices of $\Gamma^4$ are connected by an edge if the corresponding four-cycles in $\Gamma$ share a pair of non-adjacent vertices. Given a subgraph $K$ of $\Gamma^4$, we define the \emph{support} of $K$ to be the collection of vertices of $\Gamma$ (i.e. generators of $G_\Gamma$) that appear in the four-cycles in $\Gamma$ corresponding to the vertices of $K$. A graph $\Gamma$ is \emph{$\mathcal{CFS}$} if $\Gamma=\Omega*K$, where $K$ is a (possibly empty) clique and $\Omega$ is a non-empty subgraph such that $\Omega^4$ has a connected component whose support is the entire vertex set of $\Omega$.
\end{defn}

\begin{thm}[Theorem 1.1 in \cite{MR3314816}]
\label{dt}
Let $\Gamma$ be a finite, simplicial, connected, triangle-free graph which has no separating vertices or edges. Let $G_\Gamma$ be the associated right-angled Coxeter group.
\begin{enumerate}
\item The group $G_\Gamma$ has linear divergence if and only if $\Gamma$ is a join.
\item The group $G_\Gamma$ has quadratic divergence if and only if $\Gamma$ is $\mathcal{CFS}$ and is not a join.
\end{enumerate}
\end{thm}

\subsubsection{Divergence of $3$--manifold groups}

Let $M$ be a compact, orientable $3$--manifold with empty or toroidal boundary. The $3$--manifold $M$ is \emph{geometric} if its interior admits a geometric structure in the sense of Thurston which are $3$--sphere, Euclidean $3$--space, hyperbolic $3$-space, $S^{2} \times \R$, $\mathbb{H}^{2} \times \R$, $\tilde{SL(2,\R)}$, Nil and Sol. We note that a geometric $3$--manifold $M$ is Seifert fibered if its geometry is neither Sol nor hyperbolic.
{A non-geometric $3$--manifold can be cut into hyperbolic and Seifert fibered ``blocks'' along a JSJ decomposition. It is called a \emph{graph manifold} if all the pieces are Seifert fibered, otherwise it is a \emph{mixed manifold}.}

\begin{thm}[Gersten \cite{Gersten94}, Kapovich--Leeb \cite{KL98}]
\label{thm:GKL:divergence}
Let $M$ be a non-geometric manifold. Then $M$ is a graph manifold if and only if the divergence of $\pi_1(M)$ is quadratic, and $M$ is a mixed manifold if and only if the divergence of $\pi_1(M)$ is exponential.
\end{thm}

\begin{rem}
\label{rem:seifertlinear}
Let $M$ be a compact, orientable $3$--manifold  with linear divergence. We note that $M$ is geometric, otherwise its divergence is at least quadratic. Also, $M$ is not a hyperbolic manifold because the divergence of a hyperbolic manifold is exponential. If the universal cover $\tilde{M}$ of $M$ is the direct product with $\R$ of a fattening of a tree with all vertex degrees at least $3$, then $M$ is not homeomorhic to $D^{2} \times S^1$, $T^{2} \times I$, or $K^{2} \tilde{\times} I$ (twisted I-bundle over the Klein bottle). Also $M$ is not a Sol manifold, otherwise $M$ is a closed manifold (because we excluded $D^{2} \times S^1$, $T^{2} \times I$, or $K^{2} \tilde{\times} I$) which contradicts to the fact $\tilde{M}$ is the direct product with $\R$ of a fattening of a tree with all vertex degrees at least $3$. Therefore, $M$ must be a Seifert manifold excluding $D^{2} \times S^1$, $T^{2} \times I$, or $K^{2} \tilde{\times} I$. 
\end{rem}

\subsection{Colored graphs and bisimilarity}
\label{s2}
In this section, we review the concepts of colored graphs and bisimilarity in \cite{MR2376814} and \cite{MR2727658}. We will use them to classify certain right-angled Coxeter groups in this paper.

\begin{defn}
A \emph{colored graph} is a graph $\Gamma$, a set $C$, and a ``vertex coloring'' $c:V(\Gamma) \to C$.

A \emph{weak covering} of colored graphs is a graph homomorphism $f:\Gamma \to \Gamma'$ which respects colors and has the property that for each $v\in V(\Gamma)$ and for each edge $e' \in E(\Gamma')$ at $f(v)$, there exists an $e\in E(\Gamma)$ at $v$ with $f(e)=e'$.
\end{defn}

\begin{defn}
Colored graphs $\Gamma_1$ and $\Gamma_2$ are \emph{bisimilar}, written $\Gamma_1\sim\Gamma_2$ if $\Gamma_1$ and $\Gamma_2$ weakly cover
some common colored graph.
\end{defn}

\begin{prop}[\cite{MR2376814}]
The bisimilarity relation $\sim$ is an equivalence relation. Moreover, each equivalence class has a unique minimal element up to isomorphism.
\end{prop}

\subsection{Right-angled Artin groups and connection to 3--manifold groups}

We now review the concept of right-angled Artin groups and the works of Behrstock-Neumann \cite{MR2376814} and Gordon \cite{MR2199348} on connections between right-angled Artin groups and 3--manifold groups.

\begin{defn}
Given a finite simplicial graph $\Gamma$, the associated \emph{right-angled Artin group} $A_{\Gamma}$ has generating set $S$ the vertices of $\Gamma$, and relations $st = ts$ whenever $s$ and $t$ are adjacent vertices.

\end{defn}

The following two theorems show some connections between right-angled Artin groups and 3--manifold groups.

\begin{thm}[Gordon \cite{MR2199348}]
\label{Gor}
The following are equivalent for a one-ended right-angled Artin group $A_\Gamma$:
\begin{enumerate}
\item $A_\Gamma$ is virtually a 3-manifold group;
\item $A_\Gamma$ is a 3-manifold group; and
\item $\Gamma$ is either a tree or a triangle.
\end{enumerate}
\end{thm}

\begin{thm}[Behrstock-Neumann \cite{MR2376814}]
\label{bnam}
A right-angled Artin group $A_\Gamma$ is quasi-isometric to a 3-manifold group if and only if it is a 3-manifold group (and is hence as in Theorem \ref{Gor}).
\end{thm}

\section{Graph decomposition}
\label{sec:graphdecomposition}

In this section, we study the ``tree structure'' of graphs $\Gamma$ satisfying Standing Assumptions. This structure will help us study corresponding right-angled Coxeter groups $G_\Gamma$ in next section.

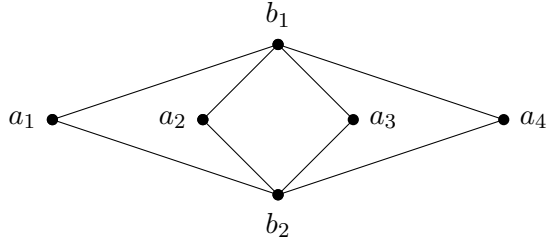
\begin{figure}
\begin{tikzpicture}[scale=1.0]

\draw (0,0) node[circle,fill,inner sep=1.5pt, color=black](1){} -- (1,1) node[circle,fill,inner sep=1.5pt, color=black](1){}-- (2,0) node[circle,fill,inner sep=1.5pt, color=black](1){}-- (1,-1) node[circle,fill,inner sep=1.5pt, color=black](1){} -- (0,0) node[circle,fill,inner sep=1.5pt, color=black](1){};

\draw (-2,0) node[circle,fill,inner sep=1.5pt, color=black](1){} -- (1,1) node[circle,fill,inner sep=1.5pt, color=black](1){}-- (4,0) node[circle,fill,inner sep=1.5pt, color=black](1){}-- (1,-1) node[circle,fill,inner sep=1.5pt, color=black](1){} -- (-2,0) node[circle,fill,inner sep=1.5pt, color=black](1){};

\node at (1,1.4) {$b_1$};\node at (1,-1.4) {$b_2$}; \node at (-2.4,0) {$a_1$};\node at (-0.4,0) {$a_2$}; \node at (2.4,0) {$a_3$};\node at (4.4,0) {$a_4$};

\end{tikzpicture}

\caption{The 4-cycle with vertices $a_2, a_3, b_1$, and $b_2$ is separating but not strongly separating with respect to the current choice of planar embedding}
\label{aja}
\end{figure}

\begin{defn}
A 4--cycle $\sigma$ of a graph $\Gamma$ \emph{separates $\Gamma$} if $\Gamma-\sigma$ has at least two components.
\end{defn}

We now talk about a stronger notion of ``separating 4-cycle'' of planar graph. This notion depends on the choice of embedding map of the ambient graph into the plane and the notion is based on Jordan Curve Theorem.

\begin{defn}
\label{de1}
Let $\Gamma$ be a graph satisfying Standing Assumptions and let $f:\Gamma\to\RR^2$ be an embedding. A 4--cycle $\sigma$ of $\Gamma$ \emph{strongly separates $\Gamma$ with respect to $f$} if $f(\Gamma)$ has non-empty intersection with both components of $\RR^2-f(\sigma)$.
\end{defn}

\begin{rem}
If the map $f$ in Definition \ref{de1} is clear from the context, we just say the 4--cycle $\sigma$ strongly separates $\Gamma$. It is clear that if a 4--cycle $\sigma$ strongly separates a graph $\Gamma$ with respect to some embedding map $f$, then $\sigma$ separates $\Gamma$ in the usual sense. However, if we fix an embedding $f$ of the graph $\Gamma$ into the plane, then a separating 4-cycle of $\Gamma$ is not necessarily strongly separating with respect to $f$. In fact, let $\Gamma$ be a planar graph with the choice of embedding $f$ in the plane as in Figure \ref{aja}, the 4-cycle with vertices $a_2, a_3, b_1$, and $b_2$ is separating but not strongly separating with respect to $f$.
\end{rem}

\begin{defn}
Assume a 4--cycle $\sigma$ strongly separates a graph $\Gamma$ with respect to an embedding $f$. Let $U_1$ and $U_2$ be two components of $\RR^2-f(\sigma)$. Let $\Gamma_i$ be $\sigma$ together with components of $\Gamma-\sigma$ that are mapped into $U_i$ via $f$. Then, $\Gamma=\Gamma_1\cup\Gamma_2$ and $\Gamma_1\cap\Gamma_2=\sigma$. We call the pair $(\Gamma_1,\Gamma_2)$ a \emph{strong visual decomposition} of $\Gamma$ along $\sigma$ with respect to $f$. If the embedding $f$ is clear from the context, we just say the pair $(\Gamma_1,\Gamma_2)$ is a strong visual decomposition of $\Gamma$ along $\sigma$ 
\end{defn}

Basically, the following lemma shows that each such subgraph $\Gamma_i$ in a strong visual decomposition of the graph $\Gamma$ above inherits important properties of the ambient graph $\Gamma$.

\begin{lem}
\label{leb}
Let $\Gamma$ be a graph satisfying Standing Assumptions. Let $(\Gamma_1,\Gamma_2)$ be a strong visual decomposition of $\Gamma$ along a 4-cycle $\sigma$ with respect to some embedding $f$. Then each subgraph $\Gamma_i$ also satisfies Standing Assumptions. Moreover, if $\Gamma$ is $\mathcal{CFS}$, then each subgraph $\Gamma_i$ is also $\mathcal{CFS}$. 
\end{lem}

\begin{proof}

It is clear that each graph $\Gamma_i$ is connected, triangle-free, planar, and has at least $5$ vertices. We now prove that if either $\Gamma_1$ or $\Gamma_2$ (say $\Gamma_1$) has a separating vertex or a separating edge $C$, then $C$ is also a separating vertex or separating edge of $\Gamma$. Let $v$ be a vertex in $\sigma-C$. Since $C$ is a separating vertex or separating edge of $\Gamma_1$, there is a vertex $u$ in $\Gamma_1-C$ such that there is no path in $\Gamma_1-C$ connecting $u$ and $v$. We observe that $\sigma-C$ is a connected set in $\Gamma_1$. Then $u$ is not a vertex of $\sigma$. We will prove that there is no path in $\Gamma-C$ connecting $u$ and $v$. Assume for the contradiction that there is a path $\alpha$ in $\Gamma-C$ connecting $u$ and $v$. We can choose a connected subpath $\beta$ of $\alpha$ connecting $u$ and some vertex $v'$ of $\sigma$ such that $\beta \cap \sigma =\{v'\}$. It is clear that $\beta$ is a path in $\Gamma_1$. Again $\sigma-C$ is a connected set in $\Gamma_1$ and two vertices $v$, $v'$ both lie in $\sigma-C$. There is a path in $\Gamma_1$ connecting $u$ and $v$ which is a contradiction. This implies that there is no path in $\Gamma-C$ connecting $u$ and $v$. Therefore, $C$ is a separating vertex or separating edge of $\Gamma$ which is a contradiction. Thus, each subgraph $\Gamma_i$ has no separating vertex and no separating edge.


We now assume that $\Gamma$ is $\mathcal{CFS}$ and we will prove that each $\Gamma_i$ is also $\mathcal{CFS}$. We only need to prove $\Gamma_1$ is $\mathcal{CFS}$ and the proof for $\Gamma_2$ is analogous. Let $K$ be a component of $\Gamma^4$ whose support is the entire vertex set of $\Gamma$. Let $K_1$ be an induced subgraph of $K$ that contains all vertices which are 4-cycles of $\Gamma_1$. It suffices to prove that $K_1$ is connected and whose support is the entire vertex set of $\Gamma_1$. 

We first prove that the 4-cycle $\sigma$ is a vertex of $K$. Let $u_1$ be a vertex in $\Gamma_1-\sigma$ and let $u_2$ be a vertex in $\Gamma_2-\sigma$. Then there is a sequence of 4--cycles $Q_1,Q_2,\cdots, Q_n$ which are vertices of $K$ such that $Q_1$ contains $u_1$ and $Q_n$ contains $u_2$ and $Q_i\cap Q_{i+1}$ is the union of two adjacent edges for each $i$. We now prove that some $Q_k$ contains two non-adjacent vertices of $\sigma$. Assume for the contradiction that no $Q_i$ contains two non-adjacent vertices of $\sigma$. Therefore, each $Q_i$ is contained in $\Gamma_1$ or $\Gamma_2$. It is clear that $Q_1$ is contained in $\Gamma_1$ and $Q_n$ is contained in $\Gamma_2$. Then there is $Q_{\ell}$ and $Q_{\ell+1}$ such that $Q_{\ell}$ is contained in $\Gamma_1$ and $Q_{\ell+1}$ is contained in $\Gamma_2$. Therefore, $Q_\ell\cap Q_{\ell+1}$ is contained in the 4-cycle $\sigma$. This implies that both $Q_{\ell}$ and $Q_{\ell+1}$ contain two non-adjacent vertices of $\sigma$ which is a contradiction. Therefore, some $Q_k$ contains two non-adjacent vertices of $\sigma$. Thus, there is a path in $\Gamma^4$ connecting $Q_k$ and $\sigma$. This implies that $\sigma$ is a vertex of $K$. Therefore, $\sigma$ is also a vertex of $K_1$.

We now prove $K_1$ is connected and it suffices to prove each vertex in $K_1$ is connected to $\sigma$ by a path in $K_1$. Let $\gamma$ be an arbitrary 4-cycle which is a vertex of $K_1$. If $\gamma$ contains two non-adjacent vertices of $\sigma$, then it is clear that there is a path in $K_1$ of length at most 2 connecting $\gamma$ and $\sigma$. Otherwise, let $\gamma=P_0,P_1,P_2,\cdots, P_m=\sigma$ be the sequence of vertices of $K$ such that $P_i\cap P_{i+1}$ is the union of two adjacent edges. Let $k$ be the smallest number such that $P_k$ contains two non-adjacent vertices of $\sigma$. Therefore, $P_i$ is contained in $\Gamma_1$ for each $i\leq k-1$. Thus, $P_i$ is a vertex in $K_1$ for each $i\leq k-1$. Let $b$ and $c$ be two non-adjacent vertices of $P_{k-1}\cap P_k$. Then it is clear that $b$ and $c$ are not non-adjacent vertices of $P_k\cap \sigma$. This implies that $P_k$ is also contained in $\Gamma_1$. Therefore, $P_k$ is also a vertex of $K_1$. Since $P_k$ contains two non-adjacent vertices of $\sigma$, there is a path of length at most 2 in $K_1$ connecting $P_k$ and $\sigma$. Thus, there is a path in $K_1$ connecting $\gamma$ and $\sigma$. Therefore, $K_1$ is connected.

We now prove that the support of $K_1$ is the entire vertex set of $\Gamma_1$. Let $u$ be a vertex in $\Gamma_1$. If $u$ is a vertex of $\sigma$ or $u$ is adjacent to non-adjacent vertices of $\sigma$, then $u$ is in the support of $K_1$ clearly. Otherwise, let $P$ be a vertex of $K$ that contains $u$. Then $P$ does not contain two non-adjacent vertices of $\sigma$. Therefore, $P$ is contained in $\Gamma_1$. Thus, $P$ is a vertex of $K_1$. Thus, $u$ belongs to the support of $K_1$. This implies that the support of $K_1$ is the entire vertex set of $\Gamma_1$. Therefore, $\Gamma_1$ is $\mathcal{CFS}$.

\end{proof}

\begin{defn}
Let $\Gamma$ be a graph satisfying Standing Assumptions and $f:\Gamma\to \RR^2$ be an embedding. We denote $n(\Gamma,f)$ the number of 4-cycles in $\Gamma$ that strongly separates $\Gamma$ with respect to $f$.

The graph $\Gamma$ is called \emph{prime} if $\Gamma$ is not a 4-cycle and $n(\Gamma,f)=0$ for some embedding $f:\Gamma\to\RR^2$.
\end{defn}

The following lemma helps us understand the structure of prime graphs.

\begin{lem}
\label{le0}
Let $\Gamma$ be a graph satisfying Standing Assumptions. Assume that $\Gamma$ is a prime graph. Then $\Gamma$ is the suspension of 3 distinct points or $\Gamma$ does not contain the suspension of 3 distinct points. In particular, if $\Gamma$ is $\mathcal{CFS}$, then it must be the suspension of 3 distinct points.
\end{lem}

\begin{proof}
We assume that $\Gamma$ contains subgraph $K$ which is a suspension of three vertices called $a_1$, $a_2$, and $a_3$. 
Let $b_1$ and $b_2$ be suspension vertices of $K$. We will show that $\Gamma=K$. Let $f:\Gamma\to\RR^2$ be an embedding. Let $C_1$ be the image of the 4-cycle with vertices $b_1$, $b_2$, $a_2$, and $a_3$. Let $C_2$ be the image of the 4-cycle with vertices $b_1$, $b_2$, $a_1$, and $a_3$. Let $C_3$ be the image of the 4-cycle with vertices $b_1$, $b_2$, $a_1$, and $a_2$. We can assume that $f(a_2)$ lies in the bounded component of $\RR^2-C_2$.

Assume for the contradiction that $\Gamma\neq K$. Then there is a vertex $d$ of $\Gamma$ that does not belong to the set $\{b_1,b_2,a_1,a_2,a_3\}$. If $f(d)$ lies in the unbounded component of $\RR^2-C_2$, then $f(\Gamma)$ intersects with both components of $\RR^2-C_2$. Therefore, the 4--cycles with vertices $b_1$, $b_2$, $a_1$, and $a_3$ strongly separates $\Gamma$ which is a contradiction. If $f(d)$ lies in the bounded component of $\RR^2-C_2$, then $f(d)$ lies in the bounded components of $\RR^2-C_1$ or $\RR^2-C_3$ (say $\RR^2-C_1$). Also $f(a_1)$ lies in the unbounded component of $\RR^2-C_1$. Therefore, $f(\Gamma)$ intersects with both components of $\RR^2-C_1$. This implies that the 4--cycles with vertices $b_1$, $b_2$, $a_2$, and $a_3$ strongly separates $\Gamma$ which is a contradiction. Therefore, $\Gamma=K$.
\end{proof}

In the following two lemmas. we discuss some behaviors of $4$--cycles in a strong decomposition of a graph.

\begin{lem}
\label{le1}
Let $\Gamma$ be a graph satisfying Standing Assumptions and $f:\Gamma\to \RR^2$ be an embedding. Assume that $(\Gamma_1,\Gamma_2$) be a strong visual decomposition of $\Gamma$ with respect to $f$ along some 4--cycle $\sigma$. Then for each $i$ the 4-cycle $\sigma$ does not strongly separates any subgraph $K$ of $\Gamma_i$ that contains $\sigma$ with respect to $f_{|K}$. Moreover, if a 4-cycle $\alpha$ in some $\Gamma_i$ that strongly separates $\Gamma_i$ with respect to $f_{|\Gamma_i}$, then $\alpha$ also strongly separates $\Gamma$ with respect to $f$.
\end{lem}

\begin{proof}
Let $V_{b}$ and $V_{u}$ be the two components of $\R^{2}-f(\sigma)$. By labeling, we assume that $f(\Gamma_1) \subset V_{b} \cup f(\sigma)$ and $f(\Gamma_2) \subset V_{u} \cup f(\sigma)$.
Let $K$ be any subgraph of $\Gamma_{i}$ such that $K$ contains $\sigma$. We will show that $\sigma$ does not strongly separate $K$ with respect to $f_{|K}$. Without losing generality, we can assume that $i=1$ (the case $i$=2 is similar). It follows that $f(K) \subset f(\Gamma_1)$. We now show that $f(K) \cap V_{u} = \emptyset$. In deed, we know that $f(\Gamma_{2}) - f(\sigma) \subset V_{u}$ and $f(\sigma) = f(\Gamma_1) \cap f(\Gamma_2)$. It follows that $f(K) \cap \bigl (f(\Gamma_{2}) - f(\sigma) \bigr ) \subset f(\Gamma_1) \cap \bigl (f(\Gamma_{2}) - f(\sigma) \bigr ) = \emptyset$, thus $f(K) \cap V_{u} = \emptyset$ because $f(K) \cap V_{u} = f(K) \cap \bigl (f(\Gamma_{2}) - f(\sigma) \bigr )$. 

We are now going to prove that if $\alpha$ is a $4$--cycle in some $\Gamma_{i}$ which is strongly separates $\Gamma_i$ with respect to $f_{|\Gamma_i}$, then $\alpha$ also strongly separates $\Gamma$ with respect to $f$. Let $U_{u}$ and $U_{b}$ be two components of $\R^{2} - f(\alpha)$. Since $\alpha$ is strongly separating $\Gamma_{i}$ with respect to $f_{|\Gamma_i}$, we have $f(\Gamma_{i}) \cap U_{b}$ and $f(\Gamma_{i}) \cap U_{u}$ are non-empty set. Of course, it implies that $f(\Gamma) \cap U_{b}$ and $f(\Gamma) \cap U_{u}$ are non-empty set as well, thus $\alpha$ is strongly separating $\Gamma$ with respect to $f$.
\end{proof}

\begin{lem}
\label{le2}
Let $\Gamma$ be a graph satisfying Standing Assumptions and $f:\Gamma\to \RR^2$ be an embedding. Assume that $(\Gamma_1,\Gamma_2$) be a strong visual decomposition of $\Gamma$ with respect to $f$ along some 4--cycle $\sigma$. If $\alpha$ is a 4-cycle that does not strongly separates $\Gamma$ with respect to $f$, then $\alpha$ is contained in $\Gamma_1$ or $\Gamma_2$.
\end{lem}

\begin{proof}
If $\alpha\cap\sigma$ does not contain two non-adjacent vertices, then $\alpha$ is contained in $\Gamma_1$ or $\Gamma_2$ clearly. We now assume that $\alpha\cap\sigma$ contains two non-adjacent vertices. Let $(a_1, a_2)$ and $(b,c)$ be two pairs of non-adjacent vertices of $\sigma$. Let $(a_3, a_4)$ and $(b,c)$ be two pairs of non-adjacent vertices of $\alpha$. Assume for the contradiction that $\alpha$ is not contained in $\Gamma_1$ or $\Gamma_2$. Then $f(a_3)$ and $f(a_4)$ lie in different components of $\RR^2-f(\sigma)$. Therefore, $f(a_1)$ and $f(a_2)$ lie in different components of $\RR^2-f(\alpha)$. This implies that $\alpha$ strongly separates $\Gamma$ with respect to $f$ which is a contradiction. Therefore, $\alpha$ is contained in $\Gamma_1$ or $\Gamma_2$. 
\end{proof}

The following lemma is a key step to decompose a graph satisfying Standing Assumptions into a tree of subgraphs.

\begin{lem}
\label{le3}
Let $\Gamma$ be a graph satisfying Standing Assumptions and $f:\Gamma\to \RR^2$ be an embedding. Assume there is a finite tree $T$ that encodes the structure of $\Gamma$ as follows:
\begin{enumerate}
\item Each vertex $v$ of $T$ is associated to an induced connected subgraph $\Gamma_v$ of $\Gamma$ that satisfies Standing Assumptions. Moreover, $\Gamma_v\neq\Gamma_{v'}$ if $v\neq v'$ and $\bigcup_{v\in V(T)}^{}\Gamma_v=\Gamma$.
\item Each edge $e$ of $T$ is associated to a 4--cycle $\Gamma_e$ of $\Gamma$. Moreover, $\Gamma_e\neq\Gamma_{e'}$ if $e\neq e'$.
\item Two vertices $v_1$ and $v_2$ of $T$ are endpoints of the same edge $e$ if and only if $\Gamma_{v_1}\cap\Gamma_{v_2}=\Gamma_e$. Moreover, if $V_1$ and $V_2$ are vertex sets of two components of $T$ removed the midpoint of $e$, then $(\bigcup_{v\in V_1}\Gamma_v,\bigcup_{v\in V_2}\Gamma_v)$ is a strong visual decomposition of $\Gamma$ along $\Gamma_e$ with respect to $f$.
\item The number $m=\max_{v\in V(T)}(n(\Gamma_v,f_{|\Gamma_v})$ is positive.
\end{enumerate} 

Then there is another tree $\bar{T}$ that encodes the structure of $\Gamma$ as in Conditions (1), (2), and (3) as above and $n(\Gamma_v,f_{|\Gamma_v})\leq m-1$ for each vertex $v$ of $\bar{T}$. Moreover, if subgraph $\Gamma_v$ is $\mathcal{CFS}$ for each vertex $v$ of $T$, then subgraph $\Gamma_w$ is also $\mathcal{CFS}$ for each vertex $w$ of $\bar{T}$.
\end{lem}

\begin{proof}
Let $v_0$ be an arbitrary vertex of $T$ such that $m=n(\Gamma_{v_0},f_{|\Gamma_{v_0}})$. Since $n(\Gamma_{v_0},f_{|\Gamma_{v_0}})>0$, the graph $\Gamma_{v_0}$ has a 4-cycle $\sigma$ that strongly separates $\Gamma_{v_0}$ with respect to $f_{|\Gamma_{v_0}}$. Let $(\Gamma_1,\Gamma_2)$ be a strong visual decomposition of $\Gamma_{v_0}$ along $\sigma$ with respect to $f_{|\Gamma_{v_0}}$. Let $e$ be an arbitrary edge of $T$ that contains $v_0$ as an endpoints. Then the 4-cycle $\Gamma_e$ does not strongly separates $\Gamma_{v_0}$ with respect to $f_{|\Gamma_{v_0}}$ by Lemma \ref{le1}. Therefore, the 4-cycle $\Gamma_e$ is contained in $\Gamma_1$ or $\Gamma_2$ by Lemma \ref{le2}. Thus, we can modify the tree $T$ to obtain another tree $T'$ as follows. 

We first replace the vertex $v_0$ of $T$ by an edge $e_0$ with two endpoints $v_1$ and $v_2$. We associate the new edge $e_0$ to the 4--cycle $\Gamma_{e_0}=\sigma$. We associate the new vertex $v_1$ the graph $\Gamma_{v_1}=\Gamma_1$ and each edge $e$ of $T$ satisfying $\Gamma_e\subset \Gamma_1$ is attached to $v_1$ in the new tree $\bar{T}$. Similarly, we associate the new vertex $v_2$ the graph $\Gamma_{v_2}=\Gamma_2$ and each edge $e$ of $T$ satisfying $\Gamma_e\subset \Gamma_2$ is attached to $v_2$ in the new tree $\bar{T}$. It is not hard to see the new tree $\bar{T}$ encodes the structure of the graph $\Gamma$ carrying Conditions (1), (2), and (3) in the lemma. Moreover, the numbers $n(\Gamma_{v_1},f_{|\Gamma_{v_1}})$ and $n(\Gamma_{v_2},f_{|\Gamma_{v_2}})$ is less than or equal to $m-1$ by Lemma \ref{le1} and the number $n(\Gamma_{v},f_{|\Gamma_{v}})$ does not change for other vertices. Also the new vertex graphs $\Gamma_{v_1}$ and $\Gamma_{v_2}$ also satisfy Standing Assumptions by Lemma \ref{leb}. Also by this lemma, two new vertex graphs $\Gamma_{v_1}$ and $\Gamma_{v_2}$ are $\mathcal{CFS}$ if $\Gamma_{v_0}$ is $\mathcal{CFS}$. Repeating this process to any vertex $v$ satisfying $n(\Gamma_{v},f_{|\Gamma_{v}})=m$, we can obtained the desired tree $\bar{T}$. Moreover, if subgraph $\Gamma_v$ is $\mathcal{CFS}$ for each vertex $v$ of $T$, then subgraph $\Gamma_w$ is also $\mathcal{CFS}$ for each vertex $w$ of $\bar{T}$.
\end{proof}

The following proposition is a direct result of Lemma \ref{le3}.

\begin{prop}
\label{prop:keyidea1}
Let $\Gamma$ be a graph satisfying Standing Assumptions and $f:\Gamma\to \RR^2$ be an embedding. Then there is a finite tree $T$ that encodes the structure of $\Gamma$ as follows:
\begin{enumerate}
\item Each vertex $v$ of $T$ is associated to an induced prime subgraph $\Gamma_v$ of $\Gamma$. Moreover, $\Gamma_v\neq\Gamma_{v'}$ if $v\neq v'$ and $\bigcup_{v\in V(T)}^{}\Gamma_v=\Gamma$.
\item Each edge $e$ of $T$ is associated to a 4--cycle $\Gamma_e$ of $\Gamma$. Moreover, $\Gamma_e\neq\Gamma_{e'}$ if $e\neq e'$.
\item Two vertices $v_1$ and $v_2$ of $T$ are endpoints of the same edge $e$ if and only if $\Gamma_{v_1}\cap\Gamma_{v_2}=\Gamma_e$. Moreover, if $V_1$ and $V_2$ are vertex sets of two components of $T$ removed the midpoint of $e$, then $(\bigcup_{v\in V_1}\Gamma_v,\bigcup_{v\in V_2}\Gamma_v)$ is a strong visual decomposition of $\Gamma$ along $\Gamma_e$ with respect to $f$.
\end{enumerate}
Moreover, if the graph $\Gamma$ is $\mathcal{CFS}$, then subgraph $\Gamma_v$ is also $\mathcal{CFS}$ for each vertex $v$ of $T$ (therefore, $\Gamma_v$ is a suspension of exactly three points by Lemma \ref{le0}).
\end{prop}

Using the ``tree structure'' on a defining graph $\Gamma$ as in Proposition \ref{prop:keyidea1} can help us understand the structure of the corresponding right-angled Coxeter group $G_\Gamma$.

\begin{cor}
\label{nicecor}
Let $\Gamma$ be a graph satisfying Standing Assumptions. Then the right-angled Coxeter group $G_\Gamma$ is a tree of groups that satisfies the following conditions:
\begin{enumerate}
\item Each vertex group $T_v$ is $G_C$ where $C$ is the suspension of three distinct points or $T_v$ is a relatively hyperbolic group with respect to a collection of $D_\infty\times D_\infty$ subgroups of $T_v$.
\item Each edge group is $D_\infty\times D_\infty$.
\end{enumerate}
Moreover, all vertex groups are isomorphic to a right-angled Coxeter group of the suspension of three distinct points if and only if $\Gamma$ is $\mathcal{CFS}$.
\end{cor}

\begin{proof}
We decompose the defining graph $\Gamma$ as a tree $T$ of subgraphs as in Proposition~\ref{prop:keyidea1}. This decomposition induces the corresponding decomposition of $G_\Gamma$ as a tree of groups. Since each edge graph in Proposition~\ref{prop:keyidea1} is an induced $4$--cycle, each edge group in the corresponding decomposition of $G_\Gamma$ is $D_\infty\times D_\infty$ which proves the Statement (2). We now prove the Statement (1).

Let $v$ be an arbitrary vertex of $T$ such that the corresponding vertex graph $\Gamma_v$ is not a suspension of three points. Therefore $\Gamma_v$ does not contain any suspension of three points by Lemma~\ref{le0}. Let $\JJ_v$ be the collection of all $4$--cycles in $\Gamma_v$. Then $\JJ_v$ satisfies Condition (1) in Theorem~\ref{th1} clearly. Since $\Gamma_v$ does not contains suspension of three points, the intersection of two $4$--cycles in $\Gamma_v$ is either empty or a point. Moreover, if a vertex $u$ of $\Gamma_v$ is adjacent to a $4$--cycle $\sigma$ of $\Gamma_v$, then $u$ must be a vertex of $\sigma$. Therefore, $\JJ_v$ satisfies Conditions (2) and (3) in Theorem~\ref{th1}. This implies that the corresponding subgroup $T_v=G_{\Gamma_v}$ is a relatively hyperbolic group with respect to a collection of $D_\infty\times D_\infty$ subgroups of $T_v$.    
\end{proof}

In the rest of this section, we will assume that the ambient graph $\Gamma$ is $\mathcal{CFS}$. Therefore, it is shown in Proposition~\ref{prop:keyidea1} that each vertex subgraph $\Gamma_v$ is a suspension of exactly three points. For our purpose of obtaining a quasi-isometric classification of right-angled Coxeter groups with $\mathcal{CFS}$ graph, the tree structure $T$ in Proposition~\ref{prop:keyidea1} is not a right one to look at. We now modify the tree $T$ to obtain a two-colored new tree that encodes structure of $\Gamma$ by doing the following construction. We refer the reader to Example~\ref{niceexample} for some explicit constructions.

\begin{cons}
\label{cons:keyidea2}
Step 1: We color an edge of $T$ by two colors: red and blue as the following. Let $e$ be an edge of $\Gamma$ with two vertices $v_1$ and $v_2$. If $\Gamma_{v_1}$ and $\Gamma_{v_2}$ have the same suspension points, then we color the edge $e$ by the red. Otherwise, we color $e$ by the blue.

Step 2: Let $\mathcal{R}$ be the union of all red edges of $T$. We remark that $\mathcal{R}$ is not necessarily connected. We form a new tree $T_{r}$ from the tree $T$ by collapsing each component $C$ of $\mathcal{R}$ to a vertex labelled by $v_C$ and we associate each such new vertex $v_C$ to the graph $\Gamma_{v_C}=\bigcup_{v\in V(C)} \Gamma_v$. For each vertex $v$ of $T_r$ which is also a vertex of $T$ we still assign $v$ the graph $\Gamma_v$ as in the previous tree $T$ structure. It is clear that for each vertex $v$ in the new tree $T_r$ vertex graph $\Gamma_{v}$ is also suspension of a vertex set called $A_{v}$. However, the number of elements in $A_v$ may be greater than three and we call this number the weight of $v$ denoted by $w(v)$. It is also clear that the new tree $T_r$ encodes the structure of $\Gamma$ carrying Conditions (1), (2), and (3) of Lemma \ref{le3}. Moreover, if $v_1$ and $v_2$ are two adjacent vertices in $T_r$, then suspension vertices of $\Gamma_{v_1}$ are elements in $A_{v_2}$ and similarly suspension vertices of $\Gamma_{v_2}$ are elements in $A_{v_1}$.

Step 3: We now choose an appropriate cyclic ordering on the set $A_v$ for vertex $v$ of $T_r$. Two vertices $a$ and $a'$ in $A_v$ are adjacent if the pair $\{a,a'\}$ together with two suspension points of $\Gamma_v$ form a 4-cycles that does not strongly separates $\Gamma_v$ with respect to $f_{|\Gamma_v}$ (see Figure \ref{aja2}). We note that if $v_1$ and $v_2$ are endpoints of an edge $e$ of $T_r$, then by Lemma \ref{le1} the 4--cycles $\Gamma_e$ does not strongly separate each graph $\Gamma_{v_i}$ with respect to $f_{|\Gamma_{v_i}}$. Therefore, suspension vertices of $\Gamma_{v_1}$ are two adjacent elements in $A_{v_2}$ and similarly suspension vertices of $\Gamma_{v_2}$ are two adjacent elements in $A_{v_1}$.

Step 4: We now color vertices of $T_r$. For each vertex $v$ of $T_r$, the graph $\Gamma_v$ is a suspension of a vertex set $A_v$ of $T_r$. We remind that the weight of $v$, denoted by $w(v)$, is the number of elements of $A_v$. It is clear that $w(v)$ is also the number of pairs of adjacent elements in $A_v$ with respect to the above cyclic ordering on $A_v$. Since for each edge $e$ of the tree $T_r$ that contains $v$ as an endpoint the $4$--cycle $\Gamma_e$ does not strongly separate $\Gamma_v$, the $4$--cycle $\Gamma_e$ contains a unique pair of non-adjacent elements of $A_v$. Moreover, if $e'$ is another edge of $T_r$ that contains $v$ as an endpoint, $\Gamma_{e'}$ must contain a different pair of non-adjacent elements of $A_v$. Therefore, the weight $w(v)$ is always greater than or equal the degree of $v$ in $T_r$. We now color $v$ by the black if its weight is strictly greater than its degree. Otherwise, we color $v$ by the white.
\end{cons}

We now summarize some key properties of the tree $T_{r}$ in the above construction:

\begin{enumerate}
\item Each vertex $v$ of $T_r$ is associated to an induced subgraph $\Gamma_v$ of $\Gamma$ that is a suspension of a vertex set $A_v$ with at least 3 elements and there is some cyclic ordering on $A_v$. We call the number of elements in $A_v$ the \emph{weight} of vertex $v$, denoted $w(v)$. 
The weight $w(v)$ of each vertex $v$ is greater than or equal its degree. We color $v$ by the black if its weight is strictly greater than its degree. Otherwise, we color $v$ by the white.
\item $\Gamma_v\neq\Gamma_{v'}$ if $v\neq v'$ and $\bigcup_{v\in V(T_r)}^{}\Gamma_v=\Gamma$.
\item Each edge $e$ of $T_r$ is associated to a 4--cycle $\Gamma_e$ of $\Gamma$. Moreover, $\Gamma_e\neq\Gamma_{e'}$ if $e\neq e'$.
\item Two vertices $v_1$ and $v_2$ of $T_r$ are endpoints of the same edge $e$ if and only if $\Gamma_{v_1}\cap\Gamma_{v_2}=\Gamma_e$. Moreover, if $v_1$ and $v_2$ are two adjacent vertices of $T_r$, suspension vertices of $\Gamma_{v_1}$ are two adjacent elements in $A_{v_2}$ and similarly suspension vertices of $\Gamma_{v_2}$ are two adjacent elements in $A_{v_1}$. Lastly, if $V_1$ and $V_2$ are vertex sets of two components of $T_r$ removed the midpoint of $e$, then $(\bigcup_{v\in V_1}\Gamma_v)\cap(\bigcup_{v\in V_2}\Gamma_v)=\Gamma_e$. 
\end{enumerate}

\begin{defn}[Visual decomposition trees]
Let $\Gamma$ be a $\mathcal{CFS}$ graph satisfying Standing Assumptions. A tree $T_r$ that encodes the structure of $\Gamma$ carrying Properties (1), (2), (3), and (4) as above is called a \emph{visual decomposition tree} of $\Gamma$.
\end{defn}

\begin{rem}
The existence of a visual decomposition tree for a $\mathcal{CFS}$ graph $\Gamma$ satisfying Standing Assumptions is guaranteed by Construction~\ref{cons:keyidea2}. We do not know whether or not the existence of visual decomposition tree for $\Gamma$ is unique. However, we only need the existence part of a such tree for our purposes. Moreover, it is not hard to draw a visual decomposition tree for a given $\mathcal{CFS}$ graph $\Gamma$ satisfying Standing Assumptions.
\end{rem}

\section{Right-angled Coxeter groups with planar defining graph}
\label{sec:racgandgraphmanifold}
In this section, we divide the collection of graphs $\Gamma$ satisfying Standing Assumptions into two types: $\mathcal{CFS}$ and non $\mathcal{CFS}$. For a $\mathcal{CFS}$ graphs $\Gamma$, we prove that the corresponding right-angled Coxeter group $G_\Gamma$ is virtually a Seifert manifold group if $\Gamma$ is a join and virtually a graph manifold group otherwise (see $(1)$ and $(2)$ in Theorem~\ref{thm:graphmanifold}). We then use the work of Behrstock-Neumann \cite{MR2376814} to classify all such groups $G_\Gamma$ up to quasi-isometry (see (\ref{haha}) in Theorem~\ref{thm:graphmanifold}). When graphs $\Gamma$ are non-join, $\mathcal{CFS}$ and satisfy Standing Assumptions, we give a characterization on $\Gamma$ for $G_\Gamma$ to be quasi-isometric to right-angled Artin groups and we also specify types of right-angled Artin groups which are quasi-isometric to such right-angled Coxeter groups (see Theorem~\ref{thm:racgraag}). For a non $\mathcal{CFS}$ graph $\Gamma$, we prove that the corresponding right-angled Coxeter groups $G_\Gamma$ is relatively hyperbolic with respect to a collection of $\mathcal{CFS}$ right-angled Coxeter subgroups of $G_\Gamma$ (see Theorem~\ref{sosonice}). These results have some applications on divergence of right-angled Coxeter groups.



\begin{figure}
\begin{tikzpicture}[scale=1.0]

\draw (0,0) node[circle,fill,inner sep=1.5pt, color=black](1){} -- (1,1) node[circle,fill,inner sep=1.5pt, color=black](1){}-- (2,0) node[circle,fill,inner sep=1.5pt, color=black](1){}-- (1,-1) node[circle,fill,inner sep=1.5pt, color=black](1){} -- (0,0) node[circle,fill,inner sep=1.5pt, color=black](1){};

\draw (-2,0) node[circle,fill,inner sep=1.5pt, color=black](1){} -- (1,1) node[circle,fill,inner sep=1.5pt, color=black](1){}-- (4,0) node[circle,fill,inner sep=1.5pt, color=black](1){}-- (1,-1) node[circle,fill,inner sep=1.5pt, color=black](1){} -- (-2,0) node[circle,fill,inner sep=1.5pt, color=black](1){};

\node at (1,1.4) {$u_1$};\node at (1,-1.4) {$u_2$}; \node at (-2.4,0) {$a$};\node at (-0.4,0) {$b$}; \node at (2.4,0) {$c$};\node at (4.4,0) {$d$}; \node at (1,1.8) {$\Gamma_v$};

\end{tikzpicture}

\caption{The graph $\Gamma_v$ is a suspension of the set $A_v=\{a,b,c,d\}$ with two suspension points $u_1$ and $u_2$. Since 4-cycles generated by $\{a,b,u_1,u_2\}$, $\{b,c,u_1,u_2\}$, $\{c,d,u_1,u_2\}$, and $\{d,a,u_1,u_2\}$ are not strongly separating, all pairs of adjacent elements in $A_v$ with respect to the cyclic ordering are $\{a,b\}$, $\{b,c\}$, $\{c,d\}$, and $\{d,a\}$.}
\label{aja2}
\end{figure}
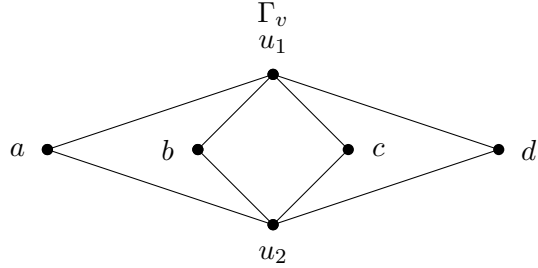

\subsection{Right-angled Coxeter groups with $\mathcal{CFS}$ graphs}
\label{nice2}
In this subsection, we will give the proof of Theorem~\ref{thm:graphmanifold} and Theorem~\ref{thm:racgraag}.

Let $\Gamma$ be a $\mathcal{CFS}$ graph satisfying Standing Assumptions. Let $T_r$ be a two-colored visual decomposition tree of $\Gamma$ (see Section~\ref{sec:graphdecomposition}). Since $\Gamma$ is planar, it follows that $G_{\Gamma}$ is virtually a $3$--manifold group. The fact $G_\Gamma$ is virtually Seifert manifold or graph manifold may not be surprising to experts. However for the purpose of obtaining a quasi-isometric classification (see (\ref{haha}) in Theorem~\ref{thm:graphmanifold}) we will construct explicitly a $3$--manifold $Y$ where the right-angled Coxeter group $G_{\Gamma}$ acts properly and cocompactly. We then elaborate the work of Kapovich-Leeb \cite{KL98}, Gersten \cite{Gersten94} to get the proof of Theorem~\ref{thm:graphmanifold}. We note that the construction of the manifold $Y$ is associated to the graph $T_{r}$, we then import the work of Behrstock-Neumann \cite{MR2376814} to get the proof of (\ref{haha}) in Theorem~\ref{thm:graphmanifold}.


\begin{cons}
\label{cons:constructionmanifold}
We now construct a 3-manifold $Y$ on which the right-angled Coxeter group $G_\Gamma$ acts properly and cocompactly. For each vertex $v$ of $T_r$, the graph $\Gamma_v$ is a suspension of a finite set $A_v$ of vertices of $\Gamma$. Let $b$ and $c$ be suspension vertices and assume $A_v$ has $n$ elements labelled cyclically by $a_i$ where $i\in \ZZ_n$. The Davis complex of the right-angled Coxeter group $G_{A_v}$ is an $n$--regular tree $T_n$ with edges labelled by $a_i$. We now construct a ``fattened tree'' $F(T_n)$ of $T_n$ as follows:

We replace each vertex of $T_n$ by a regular $n$--gon with sides labelled cyclically by $\bar{a}_i$ and we also assume the length side of the $n$--gon is $1/2$. We replace each edge $E$ labelled by $a_i$ by a strip $E\times[-1/4,1/4]$. We label each side of length $1$ of the strip $E\times[-1/4,1/4]$ by $a_i$ and we identify the edge $E$ to $E\times\{0\}$ of the strip. Moreover, if $u$ is an endpoint of the edge $E$ of $T_n$, then the edge $\{u\}\times[-1/4,1/4]$ is identified to the side labelled by $a_i$ of the $n$--gon that replaces $u$. This is clear that the right-angled Coxeter group $G_{A_v}$ acts properly and cocompactly on the fattened tree $F(T_n)$ as an analogous way its acts on the Davis complex $T_n$. By the construction, for each $i\in\ZZ_n$ there is a bi-infinite boundary geodesic, denoted $\ell_{\{i-1,i\}}$, in $F(T_n)$ that is a concatenation of edges labelled by $a_{i-1}$ and $a_i$. 

The right-angled Coxeter group $G_{\{b,c\}}$ acts on the line $\ell$ that is a concatenation of edges labelled by $b$ and $c$ by edge reflections. Let $P_v=F(T_n)\times \ell$ and we equip on $P_v$ the product metric. Then, the right-angled Coxeter group $G_{\Gamma_v}$ acts properly and cocompactly on $P_v$ in the obvious way. The space $P_v$ is also a 3-manifold with boundaries. Moreover, for each $i\in\ZZ_n$ the right-angled Coxeter groups generated by $\{a_{i-1},a_i,b,c\}$ acts on the Euclidean plane $\ell_{\{i-1,i\}}\times \ell$ as an analogous way it acts on its Davis complex. We label this plane by $\{a_{i-1},a_i,b,c\}$. 

If $v_1$ and $v_2$ are two adjacent vertices in $T_r$, then the pair of suspension vertices $(a_1,a_2)$ of $\Gamma_{v_1}$ are pair of adjacent elements in $A_{v_2}$ and the pair of suspension vertices $(b_1,b_2)$ of $\Gamma_{v_2}$ are pair of adjacent elements in $A_{v_1}$. Therefore, two spaces $P_{v_1}$ and $P_{v_2}$ have two Euclidean planes that are both labeled by $\{a_1,a_2,b_1,b_2\}$ as we constructed above. Thus, using Bass-Serre tree $\tilde{T}_r$ of the decomposition of $G_\Gamma$ as tree $T_r$ of subgroups we can form a three manifold $Y$ by gluing copies of $P_v$ appropriately and we obtain a proper, cocompact action of $G_\Gamma$ on $Y$.
\end{cons}

We first give a proof of Theorem~\ref{thm:graphmanifold}.

\begin{proof}[Proof of Theorem~\ref{thm:graphmanifold}]
Let $Y$ be the manifold in Construction~\ref{cons:constructionmanifold}. For each vertex $v$ of $T_r$, let $P_v$ be the associated space in Construction~\ref{cons:constructionmanifold}.
We now are going to prove the necessity of $(1)$ and $(2)$.
Since $\Gamma$ is $\mathcal{CFS}$, the divergence of $G_\Gamma$ is either linear or quadratic by Theorem~\ref{dt}.
If the divergence of $G_\Gamma$ is linear, then $\Gamma$ is a join $\Gamma_1*\Gamma_2$ of two induced subgraphs $\Gamma_1$ and $\Gamma_2$ by Theorem~\ref{dt}. Since $\Gamma$ is triangle-free, has at least $5$ vertices, and has no separating vertices, then each graph $\Gamma_i$ contains no edges but at least two vertices. Also $\Gamma$ is planar. Therefore, either $\Gamma_1$ or $\Gamma_2$ must contain exactly two vertices and $\Gamma$ must be a suspension of at least $3$ vertices. Therefore, the tree $T_r$ constructed as in Construction~\ref{cons:keyidea2} consists of one vertex $v$ and $G_\Gamma$ acts properly and cocompactly on $P_v$. Let $H$ be a finite index, torsion free subgroup of $G_\Gamma$. Then $H$ has linear divergence and acts freely and cocompactly on $P_v$. Therefore, $H$ is the fundamental group of the compact manifold $M=P_v/H$. By possibly passing to a finite cover of $M$, we can assume that $M$ is orientable. Moreover, the boundary components of $M$ are torus, thus $M$ is a Seifert manifold by Remark~\ref{rem:seifertlinear}.

We now assume that the divergence of $G_\Gamma$ is quadratic. Let $H$ be a finite index, torsion free subgroup of $G_\Gamma$. Then $H$ acts freely and cocompactly on the 3-manifold $Y$. Thus, $H$ is the fundamental group of the compact manifold $M=Y/H$. By possibly passing to a finite cover of $M$, we can assume that $M$ is orientable. We note that $\partial M$ consists tori. Since the divergence of $H$ is quadratic, it follows that the divergence of $\pi_1(M)$ is quadratic. It follows $M$ is a non-geometric manifold, otherwise divergence of $\pi_1(M)$ is either linear or exponential. Thus $M$ is a graph manifold by Theorem~\ref{thm:GKL:divergence}.

We are going to prove the sufficiency of $(1)$ and $(2)$. Let $\Gamma$ be just a graph satisfying Standing Assumptions. If $G_\Gamma$ is virtually a Seifert manifold group. Then the divergence of $G_\Gamma$ is linear since the divergence of a Seifert manifold group is linear. Therefore, $\Gamma$ is a join by Theorem \ref{dt}. Also, $\Gamma$ is planar and triangle-free. Therefore, $\Gamma$ is a suspension of some distinct vertices.
 
If $G_{\Gamma}$ is virtually a graph manifold group. Then the divergence of $G_\Gamma$ is quadratic since the divergence of a graph manifold group is quadratic (see Theorem~\ref{thm:GKL:divergence}). Therefore, $\Gamma$ is $\mathcal{CFS}$ and it is not a join by Theorem \ref{dt}. Again, $\Gamma$ is planar and triangle-free. Thus, $\Gamma$ is $\mathcal{CFS}$ and it is not a suspension of distinct vertices.

We are now going to prove (\ref{haha}). Since the Bass-Serre tree $\tilde{T}_r$ weakly cover $T_r$, two trees $\tilde{T}_r$ and $T_r$ are bisimilar. Also, we can color vertices of $\tilde{T}_r$ using its weakly covering on $T_r$. We observe that a vertex $v$ of $\tilde{T}_r$ is colored by black if and only if the corresponding copy of some $P_v$ includes the boundary of $Y$. Using the proof of Theorem 3.2 in \cite{MR2376814}, we obtain the proof of theorem.
\end{proof}


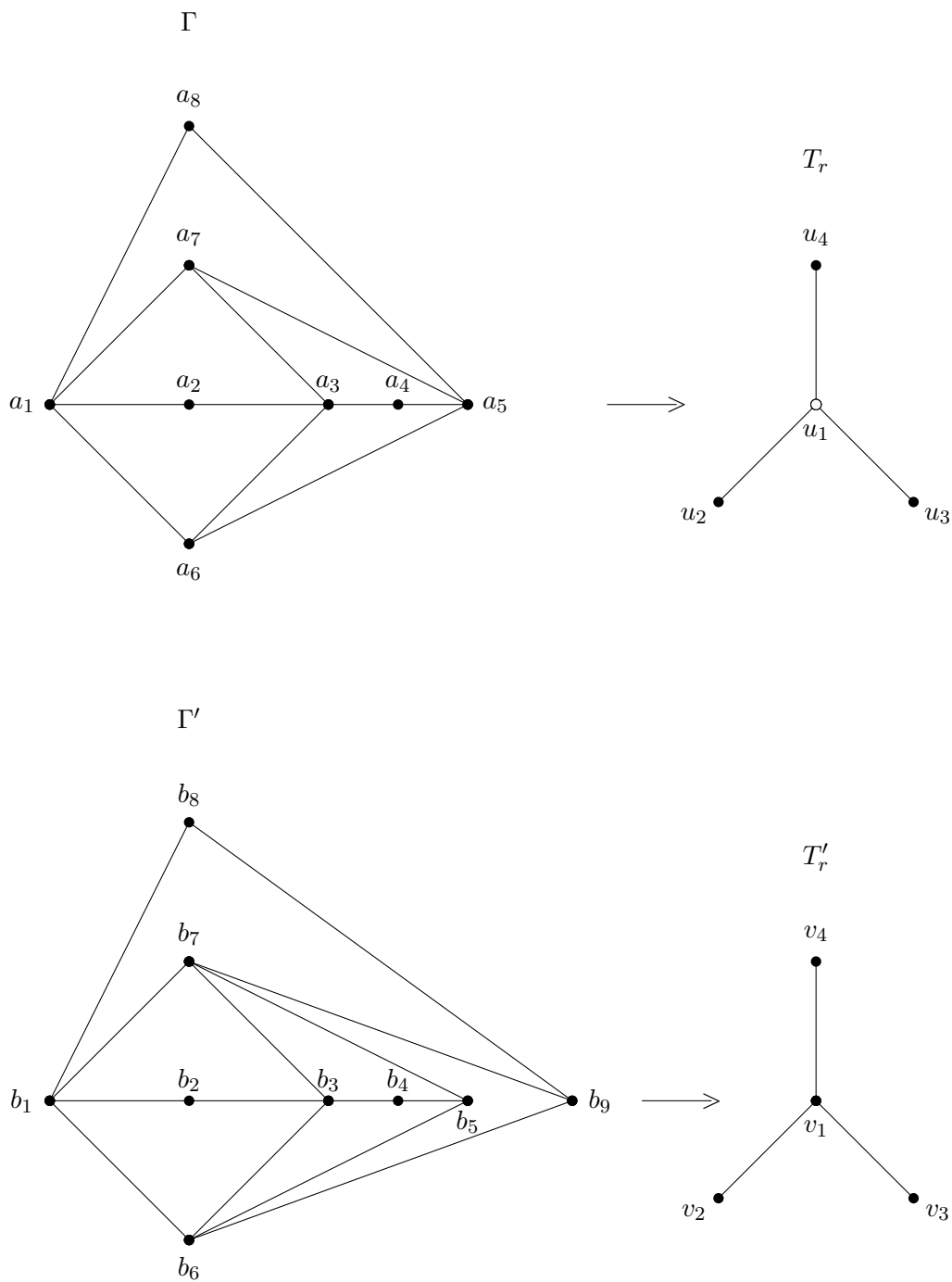
\begin{figure}
\begin{tikzpicture}[scale=1.0]

\draw (0,0) node[circle,fill,inner sep=1.5pt, color=black](1){} -- (2,0) node[circle,fill,inner sep=1.5pt, color=black](1){}-- (4,0) node[circle,fill,inner sep=1.5pt, color=black](1){}-- (5,0) node[circle,fill,inner sep=1.5pt, color=black](1){} -- (6,0) node[circle,fill,inner sep=1.5pt, color=black](1){};

\draw (2,2) node[circle,fill,inner sep=1.5pt, color=black](1){} -- (0,0) node[circle,fill,inner sep=1.5pt, color=black](1){};
\draw (2,2) node[circle,fill,inner sep=1.5pt, color=black](1){} -- (4,0) node[circle,fill,inner sep=1.5pt, color=black](1){};
\draw (2,2) node[circle,fill,inner sep=1.5pt, color=black](1){} -- (6,0) node[circle,fill,inner sep=1.5pt, color=black](1){};

\draw (2,-2) node[circle,fill,inner sep=1.5pt, color=black](1){} -- (0,0) node[circle,fill,inner sep=1.5pt, color=black](1){};
\draw (2,-2) node[circle,fill,inner sep=1.5pt, color=black](1){} -- (4,0) node[circle,fill,inner sep=1.5pt, color=black](1){};
\draw (2,-2) node[circle,fill,inner sep=1.5pt, color=black](1){} -- (6,0) node[circle,fill,inner sep=1.5pt, color=black](1){};

\draw (0,0) node[circle,fill,inner sep=1.5pt, color=black](1){} -- (2,4) node[circle,fill,inner sep=1.5pt, color=black](1){} -- (6,0) node[circle,fill,inner sep=1.5pt, color=black](1){};

\node at (-.4,0) {$a_1$};\node at (2,0.3) {$a_2$};\node at (4,0.3) {$a_3$};\node at (5,0.3) {$a_4$};\node at (6.4,0) {$a_5$}; \node at (2,-2.4) {$a_6$};\node at (2,2.4) {$a_7$};\node at (2,4.4) {$a_8$}; \draw (8,0)--(9,0);\node at (9,0) {$>$}; 

\node at (2.0,5.5) {$\Gamma$};

\draw (11,0) node[circle,draw=black,inner sep=1.5pt, fill=white](1){} -- (11,2) node[circle,fill,inner sep=1.5pt, color=black](1){};
\draw (11,0) node[circle,draw=black,inner sep=1.5pt, fill=white](1){} -- (9.6,-1.4) node[circle,fill,inner sep=1.5pt, color=black](1){};
\draw (11,0) node[circle,draw=black,inner sep=1.5pt, fill=white](1){} -- (12.4,-1.4) node[circle,fill,inner sep=1.5pt, color=black](1){};

\node at (11,-0.4) {$u_1$};\node at (9.25,-1.6) {$u_2$};\node at (12.75,-1.6) {$u_3$};\node at (11,2.4) {$u_4$};

\node at (11,3.5) {$T_r$};

\draw (0,-10) node[circle,fill,inner sep=1.5pt, color=black](1){} -- (2,-10) node[circle,fill,inner sep=1.5pt, color=black](1){}-- (4,-10) node[circle,fill,inner sep=1.5pt, color=black](1){}-- (5,-10) node[circle,fill,inner sep=1.5pt, color=black](1){} -- (6,-10) node[circle,fill,inner sep=1.5pt, color=black](1){};

\draw (2,-8) node[circle,fill,inner sep=1.5pt, color=black](1){} -- (0,-10) node[circle,fill,inner sep=1.5pt, color=black](1){};
\draw (2,-8) node[circle,fill,inner sep=1.5pt, color=black](1){} -- (4,-10) node[circle,fill,inner sep=1.5pt, color=black](1){};
\draw (2,-8) node[circle,fill,inner sep=1.5pt, color=black](1){} -- (6,-10) node[circle,fill,inner sep=1.5pt, color=black](1){};
\draw (2,-8) node[circle,fill,inner sep=1.5pt, color=black](1){} -- (7.5,-10) node[circle,fill,inner sep=1.5pt, color=black](1){};

\draw (2,-12) node[circle,fill,inner sep=1.5pt, color=black](1){} -- (0,-10) node[circle,fill,inner sep=1.5pt, color=black](1){};
\draw (2,-12) node[circle,fill,inner sep=1.5pt, color=black](1){} -- (4,-10) node[circle,fill,inner sep=1.5pt, color=black](1){};
\draw (2,-12) node[circle,fill,inner sep=1.5pt, color=black](1){} -- (6,-10) node[circle,fill,inner sep=1.5pt, color=black](1){};
\draw (2,-12) node[circle,fill,inner sep=1.5pt, color=black](1){} -- (7.5,-10) node[circle,fill,inner sep=1.5pt, color=black](1){};

\draw (0,-10) node[circle,fill,inner sep=1.5pt, color=black](1){} -- (2,-6) node[circle,fill,inner sep=1.5pt, color=black](1){} -- (7.5,-10) node[circle,fill,inner sep=1.5pt, color=black](1){};

\node at (-.4,-10) {$b_1$};\node at (2,-9.7) {$b_2$};\node at (4,-9.7) {$b_3$};\node at (5,-9.7) {$b_4$};\node at (6,-10.3) {$b_5$}; \node at (2,-12.4) {$b_6$};\node at (2,-7.6) {$b_7$};\node at (2,-5.6) {$b_8$}; \node at (7.9,-10) {$b_9$};\draw (8.5,-10)--(9.5,-10);\node at (9.5,-10) {$>$}; 

\node at (2.0,-4.5) {$\Gamma'$};

\draw (11,-10) node[circle,fill,inner sep=1.5pt, color=black](1){} -- (11,-8) node[circle,fill,inner sep=1.5pt, color=black](1){};
\draw (11,-10) node[circle,fill,inner sep=1.5pt, color=black](1){} -- (9.6,-11.4) node[circle,fill,inner sep=1.5pt, color=black](1){};
\draw (11,-10) node[circle,fill,inner sep=1.5pt, color=black](1){} -- (12.4,-11.4) node[circle,fill,inner sep=1.5pt, color=black](1){};

\node at (11,-10.4) {$v_1$};\node at (9.25,-11.6) {$v_2$};\node at (12.75,-11.6) {$v_3$};\node at (11,-7.6) {$v_4$};

\node at (11,-6.5) {$T_r'$};

\end{tikzpicture}

\caption{Two groups $G_{\Gamma}$ and $G_{\Gamma'}$ are not quasi-isometric because two corresponding decomposition trees $T_r$ and $T_r'$ are not bisimilar.}
\label{afifth}
\end{figure}

\begin{exmp}
\label{niceexample}
Let $\Gamma$ and $\Gamma'$ be graphs in Figure \ref{afifth}. It is not hard to see a visual decomposition tree $T_r$ of $\Gamma$ is shown in the same figure with the following information. Graph $\Gamma_{u_1}$ is the suspension of three vertices $a_1$, $a_3$, and $a_5$ with two suspension vertices $a_6$ and $a_7$. Graph $\Gamma_{u_2}$ is the suspension of three vertices $a_2$, $a_6$, and $a_7$ with two suspension vertices $a_1$ and $a_3$. Graph $\Gamma_{u_3}$ is the suspension of three vertices $a_4$, $a_6$, and $a_7$ with two suspension vertices $a_3$ and $a_5$. Graph $\Gamma_{u_4}$ is the suspension of three vertices $a_6$, $a_7$, and $a_8$ with two suspension vertices $a_1$ and $a_5$. We observe that each $u_i$ has weight $3$. Therefore, three vertices $u_2$, $u_3$, and $u_4$ are colored by black and $u_1$ is colored by white.

Similarly, a visual decomposition tree $T_r'$ of $\Gamma'$ is also shown in the Figure \ref{afifth} with the following information. Graph $\Gamma_{v_1}$ is the suspension of four vertices $b_1$, $b_3$, $b_5$, and $b_9$ with two suspension vertices $b_6$ and $b_7$. Graph $\Gamma_{v_2}$ is the suspension of three vertices $b_2$, $b_6$, and $b_7$ with two suspension vertices $b_1$ and $b_3$. Graph $\Gamma_{v_3}$ is the suspension of three vertices $b_4$, $b_6$, and $b_7$ with two suspension vertices $b_3$ and $b_5$. Graph $\Gamma_{v_4}$ is the suspension of three vertices $b_6$, $b_7$, and $b_8$ with two suspension vertices $b_1$ and $b_9$. We observe that each $v_i$ has weight $3$ excepts $v_1$ has weight $4$. Therefore, all four vertices $v_i$ are colored by black. Therefore, two visual decomposition trees $T_r$ and $T_r'$ are not bisimilar although they are isomorphic if we ignore the vertex colors. Therefore, two groups $G_\Gamma$ and $G_{\Gamma'}$ are not quasi-isometric.
\end{exmp} 

We now discuss connection between right-angled Coxeter groups $G_\Gamma$ of non-join, $\mathcal{CFS}$ graphs $\Gamma$ satisfying Standing Assumptions and right-angled Artin groups. 


\begin{proof}[Proof of Theorem~\ref{thm:racgraag}]
We first prove that two statements (1) and (2) are equivalent and it suffices to prove that statement (1) implies statement (2). Assume the right-angled Coxeter group $G_\Gamma$ is quasi-isometric to a right-angled Artin group $A_\Omega$. Then $A_\Omega$ is one-ended and quasi-isometric to a 3--manifold group by Theorem \ref{thm:graphmanifold}. Therefore, $A_\Omega$ is a one-ended, 3--manifold group by Theorem \ref{bnam}. Thus, $\Omega$ is a tree or a triangle by Theorem \ref{Gor}. Since $G_\Gamma$ is virtually a graph manifold group by Theorem \ref{thm:graphmanifold}, the graph $\Omega$ must be a tree of diameter at least 3. Therefore, two statements (1) and (2) are equivalent. 

The equivalence between two statement (2) are (3) are proved by Behrstock-Neumann in \cite{MR2376814}. We now prove that two statements (3) and (4) are equivalent. We first prove statement (3) implies statement (4). Assume that the right-angled Coxeter group $G_\Gamma$ is quasi-isometric to the right-angled Artin group $A_\Omega$ of a tree $\Omega$ of diameter exactly 3. We now assume for the contradiction that the tree $T_r$ contains a white vertex. As we discussed above, $G_\Gamma$ is virtually a fundamental group of a graph manifold $M$ such that $M$ has at least one Seifert component that does not contain any boundary component of $M$. Therefore, the group $A_\Omega$ is quasi-isometric to $\pi_1(M)$. On the other hand, Behrstock-Neumann in \cite{MR2376814} shows that $A_\Omega$ is the fundamental group of a graph manifold $M'$ with boundary components in each Seifert piece and the fundamental group of the such manifold $M'$ is not quasi-isometric to $\pi_1(M)$, this is a contradiction. Therefore, all vertices of the tree $T_r$ are black.

We now prove that statement (4) implies statement (3). In fact, if all vertices of the tree $T_r$ are black, the group $G_\Gamma$ is virtually the fundamental group of a graph manifold $M_1$ with boundary components in each Seifert piece. Also, the right-angled Artin group $A_\Omega$ of a tree $\Omega$ of diameter exactly 3 is the fundamental group of a graph manifold $M_2$ with boundary components in each Seifert piece. Moreover, two groups $\pi_1(M_1)$ and $\pi_1(M_2)$ are quasi-isometric by Behrstock-Neumann \cite{MR2376814}. Therefore, the right-angled Coxeter group $G_\Gamma$ is quasi-isometric to the right-angled Artin group $A_\Omega$. 
\end{proof}

\subsection{Right-angled Coxeter groups with non-$\mathcal{CFS}$ graphs}
\label{nice3}

In this subsection, we are going to prove Theorem~\ref{sosonice}.

Let $\Gamma$ be a non $\mathcal{CFS}$ graph satisfying Standing Assumptions. Let $f:\Gamma\to \RR^2$ be an embedding. Let $T$ be a tree that encodes the structure of $\Gamma$ as in Proposition \ref{prop:keyidea1}. Since $\Gamma$ is not a $\mathcal{CFS}$ graph, there is a vertex $v_0$ of $T$ such that $\Gamma_{v_0}$ does not contain a suspension of three points.

For each adjacent edge $e$ of $v_0$ let $V_e^1$ and $V_e^2$ be vertex sets of two components of $T$ removed the midpoint of $e$ and we assume that $V_e^2$ contains the vertex $v_0$. Let $K_e=\bigcup_{v\in V_e^1}\Gamma_v$ and $L_e=\bigcup_{v\in V_e^2}\Gamma_v$. Then $K_e\cap L_e=\Gamma_e$ by Proposition \ref{prop:keyidea1}.

Let $e_1$ and $e_2$ be two arbitrary adjacent edges of $v_0$. Then it is clear that $V_{e_1}^1\subset V_{e_2}^2$ and $V_{e_2}^1\subset V_{e_1}^2$. Therefore, $K_{e_1}\subset L_{e_2}$ and $K_{e_2}\subset L_{e_1}$. Therefore, $K_{e_1}\cap	K_{e_2}\subset L_{e_2}\cap	K_{e_2}\subset \Gamma_{e_2}$. Similarly, we also have $K_{e_1}\cap	K_{e_2}\subset \Gamma_{e_1}$. This implies that $K_{e_1}\cap	K_{e_2}\subset \Gamma_{e_1}\cap \Gamma_{e_2}$. Also $\Gamma_{e_1}$ and $\Gamma_{e_2}$ are both 4-cycles in $\Gamma_{v_0}$ which does not contain a suspension of three points. Thus, $\Gamma_{e_1}\cap	\Gamma_{e_2}$ is empty or a vertex or an edge. Therefore, $K_{e_1}\cap	K_{e_2}$ is empty or a vertex or an edge.

Let $\JJ_{v_0}^1$ be the collection of all graphs $K_e$ for edges $e$ adjacent to $v_0$. Then $\JJ_{v_0}^1$ satisfies Condition (2) of Theorem \ref{th1} by the above argument. Let $\JJ_{v_0}^2$ be the collection of all 4-cycles in $\Gamma_{v_0}$ which are distinct from $\Gamma_e$ for adjacent edge $e$ of $v_0$. Since $\Gamma_{v_0}$ which does not contain a suspension of three points, $\JJ_{v_0}^2$ also satisfies Condition (2) of Theorem \ref{th1}. Let $\JJ_{v_0}=\JJ_{v_0}^1\cup\JJ_{v_0}^2$.

We use the following proposition in the proof of Theorem~\ref{sosonice}.
\begin{prop}
\label{nice1}
The right-angled Coxeter group $G_\Gamma$ is relatively hyperbolic with respect to the collection $\PP_{v_0}=\set{G_J}{J\in \JJ_{v_0}}$.
\end{prop}
\begin{proof}

We will prove that $\JJ_{v_0}$ also satisfies Condition (2) of Theorem \ref{th1}. It suffices to show the intersection between a graph $K_e$ in $\JJ_{v_0}^1$ and a 4-cycle $\sigma$ in $\JJ_{v_0}^2$ is empty or a vertex or an edge. Indeed, $K_e\cap \sigma=K_e \cap (\Gamma_{v_0}\cap \sigma)=(K_e\cap \Gamma_{v_0})\cap \sigma=\Gamma_e\cap \sigma$ which is empty or a vertex or an edge since $\Gamma_{v_0}$ does not contain a suspension of three points. Therefore, $\JJ_{v_0}$ satisfies Condition (2) of Theorem \ref{th1}.

We now prove that $\JJ_{v_0}$ satisfies Condition (3) of Theorem \ref{th1}. We first prove that $\JJ_{v_0}^2$ satisfies Condition (3) of Theorem \ref{th1}. Let $\sigma$ be a 4-cycle in $\JJ_{v_0}^2$ and $d$ be a vertex that is adjacent to non-adjacent vertices $b$ and $c$ of $\sigma$. We now prove that $d$ is a vertex of $\Gamma_{v_0}$. Assume for the contradiction that $d$ does not belong to $\Gamma_{v_0}$. Therefore, $d$ is a vertex of $K_e-\Gamma_e$ for some adjacent edge $e$ of $v_0$. Since $\Gamma_e \cap \sigma$ does not contain non-adjacent vertices, either $b$ or $c$ (say $b$) does not belong to $\Gamma_e$. Therefore, two vertices $b$ and $d$ lies in the same component of $\Gamma-\Gamma_e$. This implies that $K_e-\Gamma_e$ and $\Gamma_{v_0}-\Gamma_e$ are contained in the same component of $\Gamma-\Gamma_e$ which is a contradiction. Therefore, $d$ is a vertex of $\Gamma_{v_0}$. Since $\Gamma_{v_0}$ does not contain a suspension of three points, then $d$ is a vertex of $\sigma$. Therefore, $\JJ_{v_0}^2$ satisfies Condition (3) of Theorem \ref{th1}. 
We now prove that $\JJ_{v_0}^1$ satisfies Condition (3) of Theorem \ref{th1}. Let $K_e$ be a subgraph in $\JJ_{v_0}^1$ and $d$ a vertex that are adjacent to non-adjacent vertices $b$ and $c$ of $K_e$. Assume for the contradiction that $d$ is not a vertex $K_e$. Using a similar argument as above, two points $b$ and $c$ are vertices of $\Gamma_e$. Therefore, if $d$ is a vertex of $\Gamma_{v_0}$, then $\Gamma_{v_0}$ contains a suspension of three points which is a contradiction. Thus, $d$ is a vertex of $K_{e_1}-\Gamma_{e_1}$ for some adjacent edge $e_1$ of $v_0$ other than $e$. Also $K_e\subset L_{e_1}$ as we observe above, then two points $b$ and $c$ are vertices of $L_{e_1}$. Therefore, using a similar argument as above, two points $b$ and $c$ are vertices of $\Gamma_{e_1}$. Therefore, $\Gamma_e\cap\Gamma_{e_1}$ contains two non-adjacent vertices $b$ and $c$. This implies that $\Gamma_{v_0}$ contains a suspension of three points which is a contradiction. Therefore, $\JJ_{v_0}^1$ satisfies Condition (3) of Theorem \ref{th1}. Thus, $\JJ_{v_0}$ satisfies Condition (3) of Theorem \ref{th1}.

Finally, we prove that $\JJ_{v_0}$ satisfies Condition (1) of Theorem \ref{th1}. Let $\sigma$ be an arbitrary 4--cycle of $\Gamma$. It is clear that if $\sigma\cap \Gamma_e$ does not contains non-adjacent vertices for all adjacent edge $e$ of $v_0$, then $\sigma$ is either a
4--cycle in $\JJ_{v_0}^2$ or a 4-cycle in a subgraph of $\JJ_{v_0}^1$. Now we assume that there is an adjacent edge $e$ of $v_0$ such that $\sigma\cap \Gamma_e$ contains two non-adjacent vertices $b_1$ and $b_2$. Let $a_1$ and $a_2$ be the remaining vertices of $\sigma$. Since both $a_1$ and $a_2$ are adjacent to both vertices of $K_e$, they are all vertices of $K_e$ as we prove above. Thus, $\sigma$ is a 4--cycle of $K_e$. Therefore, $\JJ_{v_0}$ satisfies Condition (1) of Theorem \ref{th1}.
\end{proof}

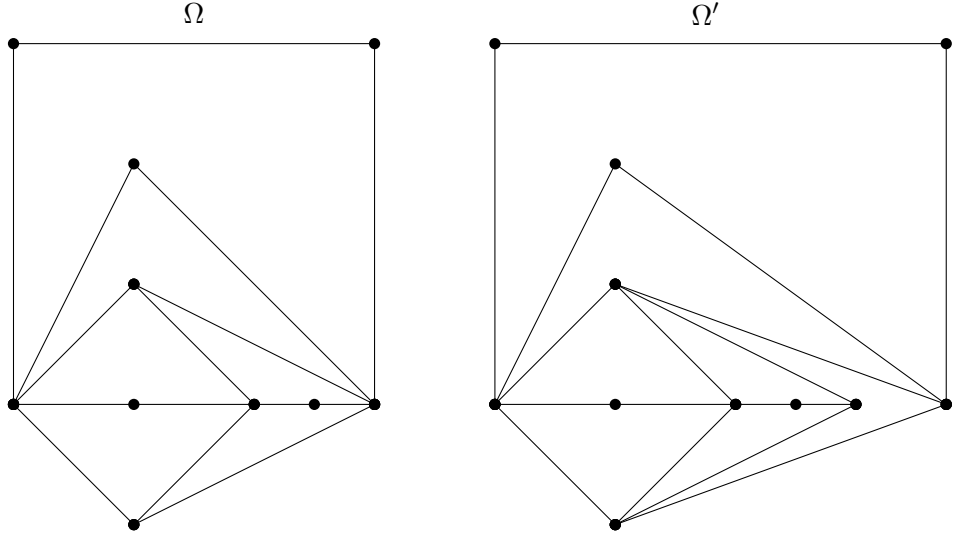
\begin{figure}
\begin{tikzpicture}[scale=0.8]

\draw (0,0) node[circle,fill,inner sep=1.5pt, color=black](1){} -- (2,0) node[circle,fill,inner sep=1.5pt, color=black](1){}-- (4,0) node[circle,fill,inner sep=1.5pt, color=black](1){}-- (5,0) node[circle,fill,inner sep=1.5pt, color=black](1){} -- (6,0) node[circle,fill,inner sep=1.5pt, color=black](1){};

\draw (2,2) node[circle,fill,inner sep=1.5pt, color=black](1){} -- (0,0) node[circle,fill,inner sep=1.5pt, color=black](1){};
\draw (2,2) node[circle,fill,inner sep=1.5pt, color=black](1){} -- (4,0) node[circle,fill,inner sep=1.5pt, color=black](1){};
\draw (2,2) node[circle,fill,inner sep=1.5pt, color=black](1){} -- (6,0) node[circle,fill,inner sep=1.5pt, color=black](1){};

\draw (2,-2) node[circle,fill,inner sep=1.5pt, color=black](1){} -- (0,0) node[circle,fill,inner sep=1.5pt, color=black](1){};
\draw (2,-2) node[circle,fill,inner sep=1.5pt, color=black](1){} -- (4,0) node[circle,fill,inner sep=1.5pt, color=black](1){};
\draw (2,-2) node[circle,fill,inner sep=1.5pt, color=black](1){} -- (6,0) node[circle,fill,inner sep=1.5pt, color=black](1){};

\draw (0,0) node[circle,fill,inner sep=1.5pt, color=black](1){} -- (2,4) node[circle,fill,inner sep=1.5pt, color=black](1){} -- (6,0) node[circle,fill,inner sep=1.5pt, color=black](1){};


\draw (0,0) node[circle,fill,inner sep=1.5pt, color=black](1){} -- (0,6) node[circle,fill,inner sep=1.5pt, color=black](1){}-- (6,6) node[circle,fill,inner sep=1.5pt, color=black](1){}-- (6,0) node[circle,fill,inner sep=1.5pt, color=black](1){};

\node at (3.0,6.5) {$\Omega$};




\draw (0+8,-10+10) node[circle,fill,inner sep=1.5pt, color=black](1){} -- (2+8,-10+10) node[circle,fill,inner sep=1.5pt, color=black](1){}-- (4+8,-10+10) node[circle,fill,inner sep=1.5pt, color=black](1){}-- (5+8,-10+10) node[circle,fill,inner sep=1.5pt, color=black](1){} -- (6+8,-10+10) node[circle,fill,inner sep=1.5pt, color=black](1){};

\draw (2+8,-8+10) node[circle,fill,inner sep=1.5pt, color=black](1){} -- (0+8,-10+10) node[circle,fill,inner sep=1.5pt, color=black](1){};
\draw (2+8,-8+10) node[circle,fill,inner sep=1.5pt, color=black](1){} -- (4+8,-10+10) node[circle,fill,inner sep=1.5pt, color=black](1){};
\draw (2+8,-8+10) node[circle,fill,inner sep=1.5pt, color=black](1){} -- (6+8,-10+10) node[circle,fill,inner sep=1.5pt, color=black](1){};
\draw (2+8,-8+10) node[circle,fill,inner sep=1.5pt, color=black](1){} -- (7.5+8,-10+10) node[circle,fill,inner sep=1.5pt, color=black](1){};

\draw (2+8,-12+10) node[circle,fill,inner sep=1.5pt, color=black](1){} -- (0+8,-10+10) node[circle,fill,inner sep=1.5pt, color=black](1){};
\draw (2+8,-12+10) node[circle,fill,inner sep=1.5pt, color=black](1){} -- (4+8,-10+10) node[circle,fill,inner sep=1.5pt, color=black](1){};
\draw (2+8,-12+10) node[circle,fill,inner sep=1.5pt, color=black](1){} -- (6+8,-10+10) node[circle,fill,inner sep=1.5pt, color=black](1){};
\draw (2+8,-12+10) node[circle,fill,inner sep=1.5pt, color=black](1){} -- (7.5+8,-10+10) node[circle,fill,inner sep=1.5pt, color=black](1){};

\draw (0+8,-10+10) node[circle,fill,inner sep=1.5pt, color=black](1){} -- (2+8,-6+10) node[circle,fill,inner sep=1.5pt, color=black](1){} -- (7.5+8,-10+10) node[circle,fill,inner sep=1.5pt, color=black](1){};


\draw (0+8,0) node[circle,fill,inner sep=1.5pt, color=black](1){} -- (0+8,6) node[circle,fill,inner sep=1.5pt, color=black](1){}-- (15.5,6) node[circle,fill,inner sep=1.5pt, color=black](1){}-- (15.5,0) node[circle,fill,inner sep=1.5pt, color=black](1){};

\node at (2.0+8+1.5,-4.5+10+1) {$\Omega'$};




\end{tikzpicture}

\caption{Two relatively hyperbolic right-angled Coxeter groups $G_{\Omega}$ and $G_{\Omega'}$ are not quasi-isometric because their peripheral subgroups are not quasi-isometric.}
\label{asixth}

\end{figure}

\begin{proof}[Proof of Theorem~\ref{sosonice}]
Let $T_1$ be the subgraph of $T$ induced by all vertices $v$ with $\Gamma_v$ a suspension of three points ($T_1$ is not necessarily connected). Let $\mathcal{T}$ be the set of all components of $T_1$. For each component $C$ in $\mathcal{T}$, let $\Gamma_C=\bigcup_{v\in V(C)}\Gamma_v$. Then, it is clear that $\Gamma_C$ is a $\mathcal{CFS}$ graph. Let $\JJ_1$ be the collection of all $\Gamma_C$ for all components $C$ in $\mathcal{T}$. Let $\JJ_2$ be the collection of all 4-cycles which are not part of any suspension of three vertices of $\Gamma$. Let $\JJ=\JJ_1\cup\JJ_2$. 

Let $n$ be the number of vertices $v$ of the tree $T$ such that $\Gamma_v$ is not a suspension of three points. We can prove the above proposition easily by induction on $n$ using Theorem \ref{dru} and Proposition \ref{nice1}. We leave the details to the reader.
\end{proof}




\begin{rem}
In the above theorem, we remark that if the defining graph $\Gamma$ is $\mathcal{CFS}$, then the right-angled Coxeter group $G_\Gamma$ is trivially relatively hyperbolic with respect to itself.
\end{rem}



\begin{exmp}
\label{newnew}
Let $\Omega$ and $\Omega'$ be two graphs as in Figure~\ref{asixth}. Then $\Omega$ (resp. $\Omega'$) contains subgraph $\Gamma$ (resp. $\Gamma'$) in Figure~\ref{afifth}. Moreover, group $G_{\Omega}$ (resp. $G_{\Omega'}$) is relatively hyperbolic with respect to group $G_\Gamma$ (resp. $G_{\Gamma'}$) by Theorem~\ref{th1}. However, two groups $G_\Gamma$ and $G_{\Gamma'}$ are not quasi-isometric by Example~\ref{niceexample}. Therefore, two groups $G_\Omega$ and $G_{\Omega'}$ are not quasi-isometric by Theorem 4.1 in \cite{MR2501302}.
\end{exmp}

\section{On generalization to certain high dimensional right-angled Coxeter groups}
\label{sec:generalization}

The main ingredient in the proof of quasi-isometric classification of $\mathcal{CFS}$ right-angled Coxeter groups with defining graphs satisfying Standing Assumptions is the decomposition of defining graphs as tree structures. Exploiting this idea we study a certain high dimensional right-angled Coxeter groups.

\subsection{Tree structure of the nerves of certain high dimensional RAAGs and RACGs}
\begin{figure}
\begin{tikzpicture}[scale=1.0]

\draw (3,3) node[circle,fill,inner sep=1.5pt, color=black](1){} -- (4,3) node[circle,fill,inner sep=1.5pt, color=black](1){}-- (5,3) node[circle,fill,inner sep=1.5pt, color=black](1){}-- (6,3) node[circle,fill,inner sep=1.5pt, color=black](1){}-- (7,3) node[circle,fill,inner sep=1.5pt, color=black](1){};

\node at (3,2) {$\swarrow$}; \node at (7,2) {$\searrow$};

\draw (0,0) node[circle,fill,inner sep=1.5pt, color=black](1){} -- (1,1) node[circle,fill,inner sep=1.5pt, color=black](1){}-- (2,0) node[circle,fill,inner sep=1.5pt, color=black](1){}-- (1,-1) node[circle,fill,inner sep=1.5pt, color=black](1){} -- (0,0) node[circle,fill,inner sep=1.5pt, color=black](1){};

\draw (-2,0) node[circle,fill,inner sep=1.5pt, color=black](1){} -- (1,1) node[circle,fill,inner sep=1.5pt, color=black](1){}-- (4,0) node[circle,fill,inner sep=1.5pt, color=black](1){}-- (1,-1) node[circle,fill,inner sep=1.5pt, color=black](1){} -- (-2,0) node[circle,fill,inner sep=1.5pt, color=black](1){};

\draw (-2,0) node[circle,fill,inner sep=1.5pt, color=black](1){} -- (-1,0) node[circle,fill,inner sep=1.5pt, color=black](1){}-- (0,0) node[circle,fill,inner sep=1.5pt, color=black](1){};

\draw (2,0) node[circle,fill,inner sep=1.5pt, color=black](1){} -- (3,0) node[circle,fill,inner sep=1.5pt, color=black](1){}-- (4,0) node[circle,fill,inner sep=1.5pt, color=black](1){};

\node at (-2.4,0) {$1$};\node at (0.4,0) {$1$}; \node at (1.6,0) {$1$};\node at (4.4,0) {$1$};
\node at (1,-1.4) {$0$};\node at (-0.8,-0.2) {$0$}; \node at (1,1.4) {$0$}; \node at (2.8,-0.2) {$0$}; \node at (1,2) {$K_1$};

\draw (8,0) node[circle,fill,inner sep=1.5pt, color=black](1){} -- (9,1) node[circle,fill,inner sep=1.5pt, color=black](1){}-- (10,0) node[circle,fill,inner sep=1.5pt, color=black](1){}-- (9,-1) node[circle,fill,inner sep=1.5pt, color=black](1){} -- (8,0) node[circle,fill,inner sep=1.5pt, color=black](1){};

\draw (6,0) node[circle,fill,inner sep=1.5pt, color=black](1){} -- (9,1) node[circle,fill,inner sep=1.5pt, color=black](1){}-- (12,0) node[circle,fill,inner sep=1.5pt, color=black](1){}-- (9,-1) node[circle,fill,inner sep=1.5pt, color=black](1){} -- (6,0) node[circle,fill,inner sep=1.5pt, color=black](1){};

\draw (6,0) node[circle,fill,inner sep=1.5pt, color=black](1){} -- (7,0) node[circle,fill,inner sep=1.5pt, color=black](1){}-- (8,0) node[circle,fill,inner sep=1.5pt, color=black](1){}-- (9,0) node[circle,fill,inner sep=1.5pt, color=black](1){}--(10,0) node[circle,fill,inner sep=1.5pt, color=black](1){};

\node at (5.6,0) {$1$};\node at (7.8,-0.2) {$1$}; \node at (10.3,0) {$1$};\node at (12.4,0) {$1$};
\node at (9,-1.4) {$0$};\node at (9,0.3) {$0$}; \node at (9,1.4) {$0$}; \node at (7.2,0.2) {$0$}; \node at (9,2) {$K_2$};

\node at (3,3.4) {$p_0$};\node at (4,2.6) {$f$}; \node at (5,3.4) {$p_1$};\node at (6,2.6) {$f$}; \node at (7,3.4) {$p_0$}; \node at (5,4.4) {$T$};

\end{tikzpicture}

\caption{Two non-isomorphic 1-dimensional flag complexes (triangle-free graphs) $K_1$ and $K_2$ in the collection $\KK_1$ can be constructed from the same tree $T$ in $\TT_1$.}
\label{aja'}
\end{figure}
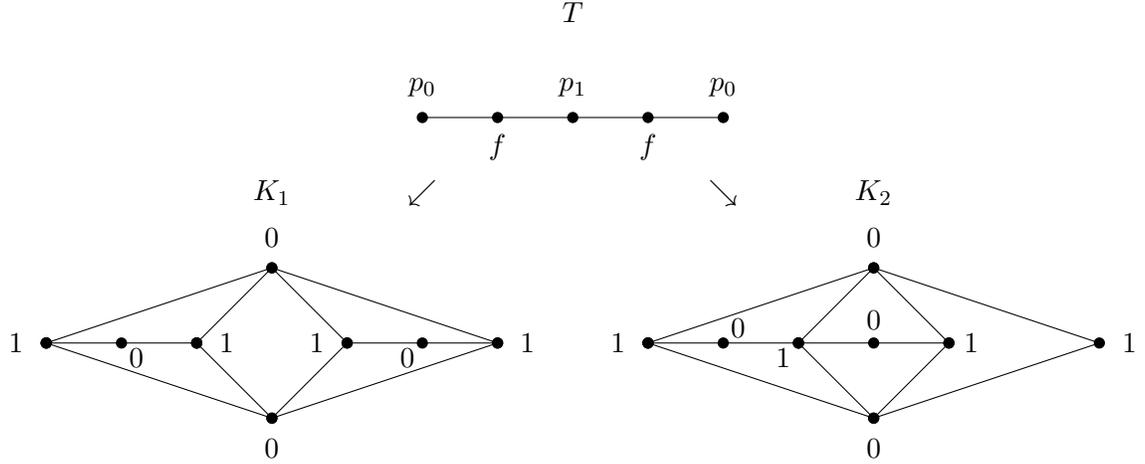

In this section, we introduce a collection of bipartite trees with certain structures and we will use this collection to construct two different collections of flag complexes. The first collection of flag complexes is used to describe high dimensional RAAGs introduced in \cite{MR2727658} and the second one is used to construct certain high dimensional RACGs. 

\begin{defn}
\label{d0}
For each integer $n\geq 1$ we define $\TT_n$ be the collection of $\pf$-bipartite tree $T$ satisfying the followings:
\begin{enumerate}
\item The valence of each $\f$-vertex is at least $2$ and at most $n+1$.
\item Each $\p$-vertex is labelled by a number in $\II_n=\{0,1,2,\cdots,n\}$ such that if $v$ and $v'$ are two different $\p$-vertices that are both adjacent to an $\f$-vertex, then $v$ and $v'$ are labelled by different numbers in $\II_n$.
\item Each $\p$-vertex $v$ is assigned to an integer $w(v)$, which we called the \emph{weight} of $v$, that is greater than or equal to the valence of $v$.
\end{enumerate}
\end{defn}

We now use each collection tree $\TT_n$ ($n\geq 1$) to construct some collection of flag complex.

\begin{defn}
\label{d1}
For each integer $n\geq 1$ and $T$ a $\pf$-bipartite tree in the collection $\TT_n$ we construct a flag complex $L\bigl(=L(T)\bigr)$ as follows:
\begin{enumerate}
\item Each $\p$-vertex $v$ of $T$ is associated to a flag complex $L_v=\Delta^{n-1}_v*B_v$, where $\Delta^{n-1}_v$ is an $(n-1)$-simplex, $B_v$ is the set of $w(v)$ distinct points, and $*$ denotes a join of two complexes. Moreover, if $v$ is labelled by a number $i$ in $\II_n$, then each point in $B_v$ is also labelled by $i$ and all $n$ vertices in $\Delta^{n-1}_v$ are labelled distinctly by elements in $\II_n-\{i\}$.
\item Each $\f$-vertex $u$ of $T$ is associated to an $n$-simplex $L_u$ and we label all $(n+1)$ vertices of $L_u$ distinctly by elements in $\II_n$.
\item If an $\f$-vertex $u$ is adjacent to a $\p$-vertex $v$, then we identify the $n$-simplex $L_u$ with an $n$-simplex in $L_v$ such that their vertex labels are matched (therefore, we have exactly $w(v)$ different ways for the identifying). Moreover, if $u$ and $u'$ are two different $\f$-vertices that are both adjacent to a $\p$-vertex $v$, then $L_{u}$ and $L_{u'}$ are identified to two different $n$-simplices of $L_v$. 
\end{enumerate}
\end{defn}

The proof for the following proposition is easy and we leave it to the reader.

\begin{prop}
Each tree $T$ in $\TT_n$ defines a unique flag complex $L(T)$ up to simplicial complex isomorphism.
\end{prop}

We now review collection of RAAG nerves studied in \cite{MR2727658}.

\begin{defn}[\cite{MR2727658}]
For each integer $n\geq 1$ we define $\LL_n$ to be the smallest class of n-dimensional simplicial complexes satisfying:
\begin{enumerate}
\item the $n$-simplex is in $\LL_n$;
\item if $L_1$ and $L_2$ are complexes in $\LL_n$ then the union of $L_1$ and $L_2$ along any $(n-1)$-simplex is in $\LL_n$.
\end{enumerate}
\end{defn}

The following proposition shows that each collection $\LL_n$ of RAAG nerves can be characterized by using the corresponding collection $\TT_n$ of bipartite trees.

\begin{prop}[\cite{MR2727658}]
For each integer $n\geq 1$, a complex $L$ belongs to the collection $\LL_n$ if and only if $L$ can be constructed from a tree $T$ in the collection $\TT_n$ as in Definition \ref{d1}. 
\end{prop}

In \cite{MR2727658}, Behrstock-Januszkiewicz-Neumann study quasi-isometry classification of collection of RAAGs $\bigl\{A_L\bigr\}_{L\in \LL_n}$ for each $n\geq 1$. 

We now discuss a different collection of simplical complexes and we will use it to introduce certain high dimensional RACGs.

\begin{defn}
\label{d2}
For each integer $n\geq 1$ and $T$ a $\pf$-bipartite tree in the collection $\TT_n$ we construct a flag complex $K\bigl(=K(T)\bigr)$ as follows:
\begin{enumerate}
\item Each $\p$-vertex $v$ of $T$ is associated to a flag complex $K_v=S^{n-1}_v*A_v$, where $S^{n-1}_v$ is an $(n-1)$-sphere $S_0*S_0*\cdots*S_0$ ($n$ factors $S_0$) and $A_v$ is the set of $w(v)$ distinct points with some cyclic ordering. Moreover, if $v$ is labelled by a number $i$ in $\II_n$, then each point in $A_v$ is labelled by $i$ and each pair of nonadjacent vertices in $S^{n-1}_v$ is labelled by the same numbers in $\II_n-\{i\}$ such that two different pairs of nonadjacent vertices in $S^{n-1}_v$ are labelled by different numbers.
\item Each $\f$-vertex $u$ of $T$ is associated to an $n$-sphere $K_u=S_0*S_0*\cdots*S_0$ ($n+1$ factors $S_0$) and we label two nonadjacent vertices in $K_u$ by the same numbers in $\II_n$ such that two different pairs of nonadjacent vertices in $K_u$ are labelled by different numbers.
\item If an $\f$-vertex $u$ is adjacent to a $\p$-vertex $v$, then we identify the complex $K_u$ with a subcomplex in $K_v$ such that their vertex labels are matched. Moreover, if the $\p$-vertex $v$ is labelled by a number $i$ in $\II_n$, then two nonadjacent vertices of the complex $K_u$ labelled by $i$ are identified to two adjacent elements in the set $A_v$ of $K_v$ with respect to the cyclic ordering on $A_v$. Lastly, if $u$ and $u'$ are two different $\f$-vertices that are both adjacent to a $\p$-vertex $v$, then $K_{u}$ and $K_{u'}$ are identified to two different subcomplexes of $K_v$.
\end{enumerate}

Let $\KK_n$ be the collection of all flag complexes each of which can be constructed from some tree in $\TT_n$ as above.
\end{defn}

\begin{rem}
We remark that two non-isomorphic flag complexes in $\KK_n$ can be constructed from the same tree $T$ in $\TT_n$ (see Figure \ref{aja'}). In this paper, we study the coarse geometry including quasi-isometry classification of collection of RACGs $\bigl\{G_K\bigr\}_{K\in \KK_n}$ for each $n\geq 1$
\end{rem}

\subsection{Quasi-isometry classification of some high dimensional right-angled Artin groups}

\label{s4}

In this section, we briefly review the work of Behrstock-Januszkiewicz-Neumann on quasi-isometry classification of RAAGs with nerves in $\LL_n$. We first review the construction of Behrstock-Januszkiewicz-Neumann of geometric models for their RAAGs.

\begin{cons}

Fix a flag complex $L$ in $\LL_n$ and we assume that $L$ can be constructed from a tree $T$ in $\TT_n$ as in Definition \ref{d1}. For each $\p$-vertex $v$ of $T$ the vertex complex $L_v=\Delta^{n-1}_v*B_v$ defines a right-angled Artin group $A_{L_{v}}$ which is the product of a free group of rank $k=w(v)$ with $\ZZ^n$.    

Let $u_1, u_2,\cdots,u_k$ be all vertices of $B_v$. Giving the free group of rank $k$ induced by $B_v$ the redundant presentation
\[\langle u_0,u_1,\cdots, u_k|u_0u_1\cdots u_k=1\rangle\]
helps us consider this free group as the fundamental group of a $(k+1)$-punctured sphere $S_{k+1}$. Therefore, the right-angled Artin group $A_{L_{v}}$ is the fundamental group of the $(n+1)$-manifold $M_v=S_{k+1}\times T^n$. Moreover, $L_{v}$ is the union of $k$ $n$-simplices of the form $\Delta^{n-1}_v*b$ ($b\in B_v$) and the right-angled Artin subgroups induced by these simplices are the fundamental groups of $k$ of the $(k+1)$ boundary components of $M_v$.

When two vertex spaces $L_{v}$ and $L_{v'}$ of $L$ intersects in an $n$-simplex, this correspond to gluing the corresponding manifolds, $M_v$ and $M_{v'}$, along a boundary component by a \emph{flip} (i.e. a map that switches the base coordinate of one piece with one of the $S^{1}$ factors of the torus fiber of the other piece). Therefore, we can associate to any flag complex $L$ in $\LL_n$ a space $X_L$ with $\pi_1(X_L)=A_L$. Thus, the right-angled Artin group $A_L$ acts properly and cocompactly on the universal cover $\tilde{X}_L$ of $X_L$. We called $\tilde{X}_L$ \emph{geometric model} of the right-angled Artin group $A_L$.

\end{cons}

By the above construction, the space $\tilde{X}_L$ can be decomposed as copies of $\tilde{M}_v=\tilde{S}_{k+1}\times \RR^n$, which we called \emph{geometric pieces} with $\p$-vertex $v$ of $T$ and they are glued accordingly. Moreover, each geometric piece have boundaries which are not shared with other geometric pieces in the decomposition.

In \cite{MR2727658}, Behrstock-Januszkiewicz-Neumann use above geometric models to classify such right-angled Artin groups $A_L$ up to quasi-isometry. Before giving a complete quasi-isometry classification for their RAAGs, for each tree $T\in \TT_n$ Behrstock-Januszkiewicz-Neumann colored it using a color set $$C_1=\{c, b_0,b_1,b_2,\cdots,b_{n-1},b_n\}$$ in the identical way of labelling vertices of $T$. More precisely, we color each $\f$-vertex by $c$ and color each $\p$-vertex labelled by $i$ in $\II_n$ by $b_i$. Although it seems to be redundant to color the tree $T$ in the way that is identical to their vertex labels, but we still want to differentiate coloring and labeling so we can compare this coloring with another coloring on $T$ we will construct later. The following theorem talks about a complete quasi-isometry classification of the collection of RAAGs $\bigl\{A_L\bigr\}_{L\in \LL_n}$ for each $n\geq 1$.

\begin{thm}[Theorem 1.1 in \cite{MR2727658}]
Let $L$ and $L'$ be two flag complexes in $\LL_n$. Assume that $L$ and $L'$ are constructed from the corresponding trees $T$ and $T'$ as in Definition \ref{d1} and we color these trees by the color set $C_1$. Then two right-angled Artin groups $A_L$ and $A_{L'}$ are quasi-isometric if and only if two trees $T$ and $T'$ are bisimilar after possibly reordering the $\p$-colors by
an element of the symmetric group on $(n+1)$ elements. 
\end{thm}

\subsection{Geometric models for high dimensional right-angled Coxeter groups with nerves in $\KK_n$ and quasi-isometry classification}
\label{s3}
In this section, we will construct a geometric model for the right-angled Coxeter group $G_K$ where $K$ is a flag complex in $\KK_n$. We then apply this geometric model and line by line argument as in Section~3 and Section~4 of \cite{MR2727658} to get the proof of Theorem~\ref{thm:quasihigherdim}. Before we construct a geometric model for $G_K$ we need a new coloring for each tree $T$ in $\TT_n$ as the following.

\underline{New coloring of each tree $T$ in $\TT_n$:}
Let $C_1$ be the color set given by Subsection~\ref{s4}. 
Let $$C_2=\{c, b_0,b_1,b_2,\cdots,b_{n-1},b_n, w_0,w_1,w_2,\cdots,w_{n-1},w_n\},$$ that contains the color set $C_1$. 

A new coloring is similar to the previous coloring except we will take vertex weight involved in the coloring process. We first color each $\f$-vertex of $T$ by $c$ in this coloring as we did with the previous coloring.

We color a $\p$-vertex as the following.
Assume that a $\p$-vertex $v$ is labelled by a number $i$ in $\II_n$. We color $v$ by $b_i$ if the weight of $v$ is strictly greater than its valence and we color $v$ by $w_i$ if the weight of $v$ is the same as its valence. Therefore, two different ways of coloring ($C_1$ and $C_2$) are identical if and only if the weight of each $\p$-vertex is strictly greater than its valence. 

\underline{Construction of geometric models:}

We now construct geometric models for our RACGs. Let $K$ be a flag complex in $\KK_n$ and we assume that $K$ can be constructed from a tree $T$ in $\TT_n$ as in Definition \ref{d2}. Let $\Sigma_K$ be the Davis complex of the right-angled Coxeter group $G_K$. We now construct a ``fatten'' Davis complex $Y_K$ on which $G_K$ acts properly and cocompact on.

For each $\p$-vertex $v$ of $T$ we have the associated flag complex $K_v=S^{n-1}_v*A_v$, where $S^{n-1}_v$ is an $(n-1)$-sphere $S_0*S_0*\cdots*S_0$ and $A_v$ is the set of $w(v)$ distinct points with some cyclic ordering. Assume that elements in $A_v$ are labelled cyclically by $a_i$ where $i\in \ZZ_{n}$ ($n=w(v)$). The Davis complex of the right-angled Coxeter group $G_{A_v}$ is an $n$--regular tree $T_n$ with edges labelled by $a_i$. We first construct a ``fattened tree'' $F(T_n)$ of $T_n$ as follows:

We replace each vertex of $T_n$ by a regular $n$--gon with sides labelled cyclically by $\bar{a}_i$ and we also assume the length side of the $n$--gon is $1/2$. We replace each edge $E$ labelled by $a_i$ by a strip $E\times[-1/4,1/4]$. We label each side of length $1$ of the strip $E\times[-1/4,1/4]$ by $a_i$ and we identify the edge $E$ to $E\times\{0\}$ of the strip. Moreover, if $w$ is an endpoint of the edge $E$ of $T_n$, then the edge $\{w\}\times[-1/4,1/4]$ is identified to the side labelled by $\bar{a}_i$ of the $n$--gon that replaces $w$. This is clear that the right-angled Coxeter group $G_{A_v}$ acts properly and cocompactly on the fattened tree $F(T_n)$ as an analogous way its acts on the Davis complex $T_n$. Moreover, the fattened tree $F(T_n)$ is a 2-dimensional manifold and each boundary component is a line which is labelled concatenatively by $\{a_{i-1},a_i\}$ for some $i\in \ZZ_n$.

The Davis complex $\Sigma_{S^{n-1}_v}$ of the right-angled Coxeter group $G_{S^{n-1}_v}$ is isometric to $\RR^n$. Let $P_v=\Sigma_{S^{n-1}_v}\times F(T_n)$. Then the right-angled Coxeter group $G_{K_v}$ acts properly and cocompactly on $P_v$ obviously. Moreover, $P_v$ is an $(n+1)$-manifold and each boundary components of $P_v$ are copies of the Davis complexes of right-angled Coxeter groups $G_{S^{n-1}_v*\{a_{i-1},a_i\}}$ ($i\in\ZZ_n$).

For each $\f$-vertex $u$ that is adjacent to a $\p$-vertex $v$ the flag complex $K_u$ is identified to a subcomplex of the form $S^{n-1}_v*\{a_{i-1},a_i\}$ in $K_v=S^{n-1}_v*A_v$. Therefore, each boundary component of $P_v$ that is a copy of the Davis complex of right-angled Coxeter group $G_{S^{n-1}_v*\{a_{i-1},a_i\}}$ can also be considered as a copy of the Davis complex $\Sigma_{K_u}$. Thus, using the Bass-Serre tree $\tilde{T}$ of the decomposition of the right-angled Coxeter group $G_K$ as the tree $T$ of subgroups we can form a space $Y_K$ by gluing copies of each space $P_v$ appropriately and we obtain a proper, cocompact action of $G_K$ on the new space $Y_K$. We call each copy of $P_v$ for some $\p$-vertex $v$ of $T$ a \emph{geometric piece of type $v$} and we call the space $Y_K$ a \emph{geometric model} for the right-angled Coxeter group $G_K$. 

\begin{rem}
\begin{enumerate}
\item We observe that for each $\p$-vertex $v$ a geometric piece of type $v$ has boundary components which are not shared with other geometric pieces if and only if the weight of the vertex $v$ is strictly greater than its valence (i.e. the vertex $v$ is colored by some color $b_i$ when we color the tree $T$ using color set $C_2$ as above).
\item We remark that the geometric model $Y_K$ of a right-angled Coxeter group $G_K$ ($K\in \KK_n$) have a similar structure with the geometric model $\tilde{X}_L$ of a right-angled Artin group $A_L$ ($L\in \LL_n$) excepts $Y_K$ may contains geometric pieces such that all its boundary components are shared with other geometric pieces. 
\end{enumerate}
\end{rem}

\begin{proof}[Proof of Theorem~\ref{thm:quasihigherdim}]
We use the geometric model $Y_K$ in the construction above for each right-angled Coxeter group $G_K$ ($K\in\KK_n$) and line by line argument as in Section 3 and Section 4 of \cite{MR2727658} to get the proof.
\end{proof}

We can also use an almost identical proof as in Sections 3 and 4 in \cite{MR2727658} to prove the following theorem:

\begin{thm}
\label{compare}
Let $L$ be a flag complex in $\LL_n$ and let $K$ be a flag complex in $\KK_n$. Assume that $L$ and $K$ can be constructed from two trees $T_L$ and $T_K$ in $\TT_n$ respectively. We color the tree $T_L$ by the color set $C_1$ and the tree $T_K$ by the color set $C_2$. Then RAAG $A_L$ and RACG $G_K$ are quasi-isometric if and only if $\p$-vertices of $T_K$ are only colored by colors in the set $C_1$ and two colored trees $T_L$ and $T_K$ are bisimilar after possibly reordering the $\p$-colors by an element of the symmetric group on $n+1$ elements.
\end{thm}

\section{Strongly quasiconvex subgroups of $\mathcal{CFS}$ right-angled Coxeter groups}
\label{sqc}

\subsection{Background on strongly quasiconvex subgroups and stable subgroups}

In this subsection, we review two notions of quasiconvex subgroups and stable subgroups. We also recall some results related to these two notions. 
\begin{defn}
A subset $A$ of a geodesic metric space $X$ is \emph{Morse} if for every $K \geq 1,C \geq 0$ there is some $M = M(K,C)$ such that every $(K,C)$--quasi–geodesic with endpoints on $A$ is contained in the $M$--neighborhood of $A$. We call the function $M$ a \emph{Morse gauge}.
\end{defn}

\begin{defn}
Let $\Phi\!:\!A\to X$ be a quasi-isometric embedding between geodesic metric spaces. We say $A$ is \emph{strongly quasiconvex} in $X$ if the image $\Phi(A)$ is Morse in $X$. We say $A$ is \emph{stable} in $X$ if for any $K\geq 1$, $L \geq 0$ there is an $R=R(K,L)\geq 0$ so that if $\alpha$ and $\beta$ are two $(K,L)$--quasi-geodesics with the same endpoints in $\Phi(A)$, then the Hausdorff distance between $\alpha$ and $\beta$ is less than $R$. 
\end{defn}

Note that when we say $A$ is strongly quasiconvex (stable) in $X$ we mean that $A$ is strongly quasiconvex (stable) in $X$ with respect to a particular quasi-isometric embedding $\Phi\!:\!A\to X$. Such a quasi-isometric embedding will always be clear from context, for example an undistorted subgroup $H$ of a finitely generated group $G$. We now recall the concepts of strongly quasiconvex subgroups and stable subgroups.

\begin{defn}
Let $G$ be a finite generated group and $S$ an arbitrary finite generating set of $G$. Let $H$ be a finite generated subgroup of $G$ and $T$ an arbitrary finite generating set of $H$. The subgroup $H$ is \emph{undistorted} in $G$ if the natural inclusion $i\!:H\to G$ induces a quasi-isometric embedding from the Cayley graph $\Gamma(H,T)$ into the Cayley graph $\Gamma(G,S)$. We say $H$ is \emph{stable} in $G$ if $\Gamma(H,T)$ is stable in $\Gamma(G,S)$. 

We remark that stable subgroups were proved to be independent of the choice of finite generating sets (see Section 3 in \cite{MR3426695}).
\end{defn}

\begin{defn}
Let $G$ be a finite generated group and $H$ a subgroup of $G$. We say $H$ is \emph{strongly quasiconvex} in $G$ if $H$ is a Morse subset in the Cayley graph $\Gamma(G,S)$ for some (any) finite generating set $S$.
\end{defn}

We remark that strongly quasiconvex subgroups were proved to be independent of the choice of finite generating sets of the ambient groups. Moreover, strongly quasiconvex subgroups are all finitely generated and undistorted. We refer the reader to the work of the second author in Section 4 in \cite{T2017} for more details. The following proposition tells us a relation between strongly quasiconvex subgroups and stable subgroups.

\begin{prop}[Proposition 4.3 in \cite{T2017}]
\label{p3}
Let $G$ be a finitely generated group. A subgroup $H$ of $G$ is stable if and only if $H$ is strongly quasiconvex and hyperbolic.
\end{prop}

The following proposition gives us a way to get another quasiconvex subgroup
from a strongly quasicovex subgroup.

\begin{prop}[Proposition 4.11 in \cite{T2017}]
\label{pp3}
Let $G$ be a finitely generated group and $A$ undistorted subgroup of $G$. If $H$ is a strongly quasiconvex subgroup of $G$, then $H_1=H\cap A$ is a strongly quasiconvex subgroup of $A$. In particular, $H_1$ is finitely generated and undistorted in $A$.
\end{prop}

We now discuss the height and the width of subgroups.

\begin{defn}
Let $G$ be a group and $H$ a subgroup.
\begin{enumerate}
\item Conjugates $g_1Hg_1^{-1}, \cdots g_kHg_k^{-1} $ are \emph{essentially distinct} if the cosets $g_1H,\cdots,g_kH$ are distinct.
\item $H$ has height at most $n$ in $G$ if the intersection of any $(n+1)$ essentially distinct conjugates is finite. The least $n$ for which this is satisfied is called the height of $H$ in $G$.
\item The \emph{width} of $H$ is the maximal cardinality of the set \[\{g_iH:\abs{g_iH{g_i}^{-1} \cap g_jH{g_j}^{-1}}=\infty\},\] where $\{g_iH\}$ ranges over all collections of distinct cosets.
\end{enumerate}
\end{defn}

We note that finite subgroups and subgroups of finite index have finite
height and width, and infinite normal subgroups of infinite index have in-
finite height and width. Hence, the next proposition states that strongly
quasiconvex subgroups are far from being normal.

\begin{thm}[Theorem 1.2 in \cite{T2017}]
\label{lacloi1}
Let $G$ be a finitely generated group and let $H$ be a strongly quasiconvex subgroup. Then $H$ has finite height and
finite width.
\end{thm}


\subsection{Strongly quasiconvex subgroups and stable subgroups in certain tree of groups and application to right-angled Coxeter groups}
\label{hihihi}

In this subsection, we prove that torsion free, strongly quasiconvex subgroups of infinite index of certain finite graph of groups are free. This result can be applied to our right-angled Coxeter groups.

\begin{lem}
\label{hihi}
Assume a group $G$ is decomposed as a finite graph $T$ of groups such that each edge group is infinite. Let $G_v$ be a vertex subgroup. Then for each $g_1$ and $g_2$ in $G$ there is a finite sequence of conjugates of vertex subgroups $g_1G_vg_1^{-1}=Q_0,Q_1,\cdots, Q_m=g_2G_vg_2^{-1}$ such that $Q_{i-1}\cap Q_i$ is infinite for each $i\in \{1,2,\cdots,m\}$. 
\end{lem}

\begin{proof}
Let $\tilde{T}$ be the Bass-Serre tree of the decomposition of $G$. Then conjugates of vertex groups (resp. edge groups) correspond to vertices (edges) of $\tilde{T}$. Therefore, the lemma follows the facts that $\tilde{T}$ is connected and each edge group is infinite. 
\end{proof}

The following proposition shows some property of strongly quasiconvex subgroups in certain graph of groups.

\begin{prop}
\label{quahay}
Assume a group $G$ is decomposed as a finite graph $T$ of groups that satisfies the following.
\begin{enumerate}
\item For each vertex $v$ of $T$ the vertex group $G_v$ is finitely generated and undistorted. Moreover, any strongly quasiconvex, infinite subgroup of $G_v$ is of finite index.
\item Each edge group is infinite.
\end{enumerate}
Then, if $H$ is a strongly quasiconvex subgroup of $G$ of infinite index, then $gHg^{-1}\cap G_v$ is finite for each vertex group $G_v$ and each group element $g$.
\end{prop}

\begin{proof}
We assume for the contradiction that $g_0Hg_0^{-1}\cap G_v$ is infinite for some vertex group $G_v$ and some $g_0\in G$. We claim that $gHg^{-1}\cap G_v$ has finite index in $G_v$ for all $g\in G$. In fact, since $g_0Hg_0^{-1}$ is a strongly quasiconvex subgroup and $G_v$ is an undistorted subgroup, then $g_0Hg_0^{-1}\cap G_v$ is a strongly quasiconvex subgroup of $G_v$ by Proposition \ref{pp3}. Therefore, $g_0Hg_0^{-1}\cap G_v$ has finite index in $G_v$ by the hypothesis. 

We now prove that $gHg^{-1}\cap G_v$ has finite index in $G_v$ for all $g\in A_\Gamma$. By Lemma \ref{hihi}, there is a finite sequence of conjugates of vertex subgroups $g_0^{-1}G_vg_0=Q_0,Q_1,\cdots, Q_m=g^{-1}G_vg$ such that $Q_{i-1}\cap Q_i$ is infinite for each $i\in \{1,2,\cdots,m\}$. Since $g_0Hg_0^{-1}\cap G_v$ has finite index in $G_v$, $H \cap g_0^{-1}G_vg_0$ has finite index in $Q_0=g_0^{-1}G_vg_0$. Also, subgroup $Q_0\cap Q_1$ is infinite. Then, $H \cap Q_1$ is infinite. Using a similar argument as above, we obtain $H \cap Q_1$ has finite index in $Q_1$. Repeating this process, we have $H\cap g^{-1}G_vg$ has finite index in $g^{-1}G_vg$. In other word, $gHg^{-1}\cap G_v$ has finite index in $G_v$. 

By Theorem \ref{lacloi1}, there is a number $n$ such that the intersection of any $(n+1)$ essentially distinct conjugates of $H$ is finite. Since $H$ has infinite index in $G$, there is $n+1$ distinct element $g_1, g_2,\cdots g_{n+1}$ such that $g_iH\neq g_j H$ for each $i\neq j$. Also, $g_i H g_i^{-1} \cap G_v$ has finite index in $G_v$ for each $i$. Then $(\cap g_i H g_i^{-1}) \bigcap G_v$ also has finite index in $G_v$. In particular, $\bigcap g_i H g_i^{-1}$ is infinite which is a contradiction. Therefore, $gHg^{-1}\cap G_v$ is finite for each vertex group $G_v$.
\end{proof}




\begin{prop}
\label{hayqua}
Assume a group $G$ is decomposed as a finite graph $T$ of groups. Let $H$ be a subgroup of $G$ such that $gHg^{-1}\cap G_v$ is trivial for each vertex group $G_v$ and each group element $g$. Then $H$ is free.
\end{prop}


\begin{proof}
Let $\tilde{T}$ be the Bass-Serre tree of the decomposition of $G$. Then $G$ acts on $\tilde{T}$ such that the stabilizer of a vertex of $T$ is a conjugate of a vertex group. To show $H$ is free, it is enough to show that $H$ acts freely on $\tilde{T}$. To see $H$ acts freely on $\tilde{T}$, it suffices to show that for each vertex $v \in \tilde{T}$ then $Stab_{H}(v) = \{e\}$. Note that $Stab_{H}(v) = Stab_{G}(v) \cap  H$. By the assumption, we have that $Stab_{G}(v) \cap  H = \{e\}$, thus $Stab_{H}(v) = \{e\}$. The proposition is proved.
\end{proof}

\begin{proof}[Proof of Proposition~\ref{vuiqua}]
The proof is a combination of Proposition~\ref{quahay} and Proposition~\ref{hayqua}.
\end{proof}

\begin{prop}
\label{pp2}
If $G$ is a finitely generated group that has infinite center and $H$ is an infinite strongly quasiconvex subgroup of $G$, then $H$ is of finite index.
\end{prop}

\begin{proof}
Let $Z$ be the center of the group $G$. We first prove that the subgroup $Z\cap H$ has finite index in $Z$. Assume for a contradiction that the subgroups $Z\cap H$ has infinite index in $Z$. Then there is an infinite sequence $(z_n)$ of elements in $Z$ such that $z_i(Z\cap H)\neq z_j(Z\cap H)$ for $i\neq j$. Therefore, $z_iH\neq z_jH$ for $i\neq j$. However, we also have $z_iHz^{-1}_i=z_jHz^{-1}_j$ for all $i\neq j$ which contradicts to Theorem 1.2 in \cite{T2017} that a strongly quasiconvex subgroup has finite height. Therefore, the subgroup $Z\cap H$ has finite index in $Z$. In particular, the subgroup $Z\cap H$ is infinite.

We now assume for a contradiction that the subgroup $H$ has infinite index in $G$. Then there is an infinite sequence $(g_n)$ of elements in $G$ such that $g_iH \neq g_j H$ for $i\neq j$. However, $Z\cap H$ is an infinite subgroup of $g_iHg^{-1}_i$ for all $i$ which contradicts to Theorem 1.2 in \cite{T2017} that a strongly quasiconvex subgroup has finite height. Therefore, the subgroup $H$ has finite index in $G$. 
\end{proof}

By combining the above proposition with Proposition~\ref{vuiqua}, we obtain the proof of Theorem \ref{good}.


\begin{proof}[Proof of Theorem~\ref{good}]
Obviously, the right-angled Coxeter group $G_K$ is a tree of groups whose vertex groups have infinite center and whose edge groups are infinite. Let $G_1$ be a finite index torsion free subgroup of the right-angled Coxeter group $G_K$ and $H_1=H\cap G_1$. Then $H_1$ is a strongly quasiconvex, torsion free subgroup of $G_K$ of infinite index. Therefore, $H_1$ is a free group by Propositions \ref{vuiqua} and \ref{pp2}. Also, $H_1$ is a finite index subgroup of $H$. Therefore, the subgroup $H$ is virtually free.
\end{proof}




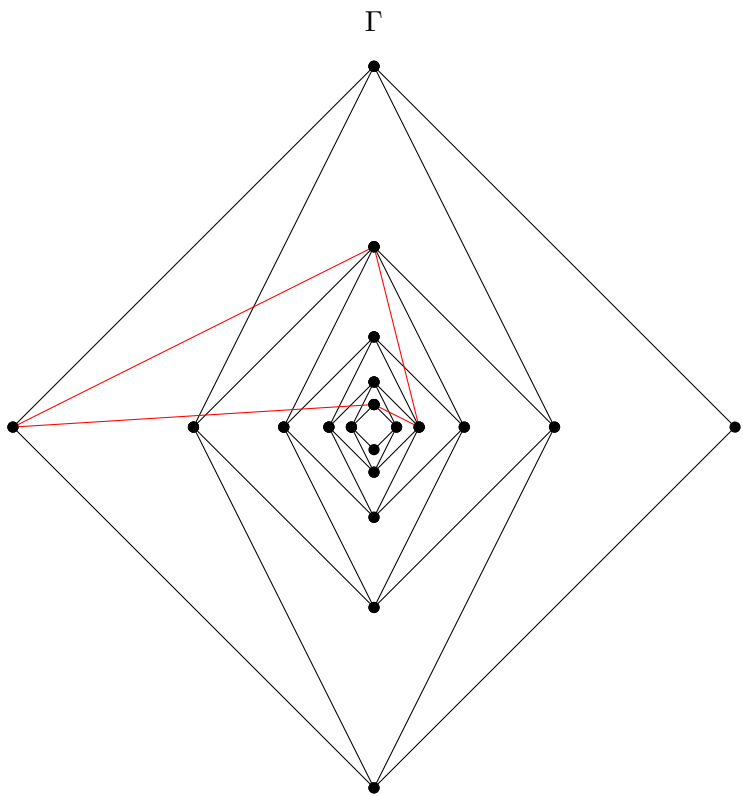
\begin{figure}
\begin{tikzpicture}[scale=.6]

\draw (0,0.5) node[circle,fill,inner sep=1.5pt, color=black](1){} -- (0.5,0) node[circle,fill,inner sep=1.5pt, color=black](1){}-- (0,-0.5) node[circle,fill,inner sep=1.5pt, color=black](1){} -- (-0.5,0) node[circle,fill,inner sep=1.5pt, color=black](1){} -- (0,0.5) node[circle,fill,inner sep=1.5pt, color=black](1){};

\draw (-0.5,0) node[circle,fill,inner sep=1.5pt, color=black](1){} -- (0,1) node[circle,fill,inner sep=1.5pt, color=black](1){}-- (0.5,0) node[circle,fill,inner sep=1.5pt, color=black](1){} -- (0,-1) node[circle,fill,inner sep=1.5pt, color=black](1){} -- (-0.5,0) node[circle,fill,inner sep=1.5pt, color=black](1){};

\draw (0,1) node[circle,fill,inner sep=1.5pt, color=black](1){} -- (1,0) node[circle,fill,inner sep=1.5pt, color=black](1){}-- (0,-1) node[circle,fill,inner sep=1.5pt, color=black](1){} -- (-1,0) node[circle,fill,inner sep=1.5pt, color=black](1){} -- (0,1) node[circle,fill,inner sep=1.5pt, color=black](1){};

\draw (-1,0) node[circle,fill,inner sep=1.5pt, color=black](1){} -- (0,2) node[circle,fill,inner sep=1.5pt, color=black](1){}-- (1,0) node[circle,fill,inner sep=1.5pt, color=black](1){} -- (0,-2) node[circle,fill,inner sep=1.5pt, color=black](1){} -- (-1,0) node[circle,fill,inner sep=1.5pt, color=black](1){};

\draw (0,2) node[circle,fill,inner sep=1.5pt, color=black](1){} -- (2,0) node[circle,fill,inner sep=1.5pt, color=black](1){}-- (0,-2) node[circle,fill,inner sep=1.5pt, color=black](1){} -- (-2,0) node[circle,fill,inner sep=1.5pt, color=black](1){} -- (0,2) node[circle,fill,inner sep=1.5pt, color=black](1){};

\draw (-2,0) node[circle,fill,inner sep=1.5pt, color=black](1){} -- (0,4) node[circle,fill,inner sep=1.5pt, color=black](1){}-- (2,0) node[circle,fill,inner sep=1.5pt, color=black](1){} -- (0,-4) node[circle,fill,inner sep=1.5pt, color=black](1){} -- (-2,0) node[circle,fill,inner sep=1.5pt, color=black](1){};

\draw (0,4) node[circle,fill,inner sep=1.5pt, color=black](1){} -- (4,0) node[circle,fill,inner sep=1.5pt, color=black](1){}-- (0,-4) node[circle,fill,inner sep=1.5pt, color=black](1){} -- (-4,0) node[circle,fill,inner sep=1.5pt, color=black](1){} -- (0,4) node[circle,fill,inner sep=1.5pt, color=black](1){};

\draw (-4,0) node[circle,fill,inner sep=1.5pt, color=black](1){} -- (0,8) node[circle,fill,inner sep=1.5pt, color=black](1){}-- (4,0) node[circle,fill,inner sep=1.5pt, color=black](1){} -- (0,-8) node[circle,fill,inner sep=1.5pt, color=black](1){} -- (-4,0) node[circle,fill,inner sep=1.5pt, color=black](1){};

\draw (0,8) node[circle,fill,inner sep=1.5pt, color=black](1){} -- (8,0) node[circle,fill,inner sep=1.5pt, color=black](1){}-- (0,-8) node[circle,fill,inner sep=1.5pt, color=black](1){} -- (-8,0) node[circle,fill,inner sep=1.5pt, color=black](1){} -- (0,8) node[circle,fill,inner sep=1.5pt, color=black](1){};

\draw (0,0.5)[color=red] node[circle,fill,inner sep=1.5pt, color=black](1){} -- (1,0) node[circle,fill,inner sep=1.5pt, color=black](1){}-- (0,4) node[circle,fill,inner sep=1.5pt, color=black](1){} -- (-8,0) node[circle,fill,inner sep=1.5pt, color=black](1){} -- (0,0.5) node[circle,fill,inner sep=1.5pt, color=black](1){};

\node at (0,9) {$\Gamma$};

\end{tikzpicture}

\caption{The special subgroup $H$ generated by the red 4-cycle is a non-stable, strongly quasiconvex subgroup of infinite index of the right-angled Coxeter group $G_\Gamma.$}
\label{afourth}
\end{figure}

\begin{exmp}
\label{hayhayqua}
We now construct a connected, triangle-free, $\mathcal{CFS}$ graph $\Gamma$ with no separating vertices or edges such that the corresponding right-angled Coxeter group $G_\Gamma$ has a non-stable, strongly quasiconvex subgroup of infinite index. 

Let $\Gamma$ be the graph in Figure \ref{afourth} and $K$ be the red 4-cycle of $\Gamma$. It is not hard to check $\Gamma$ is connected, triangle-free, $\mathcal{CFS}$ and has no separating vertices or edges. Moreover, the 4-cycle $K$ does not contain any pair of non-adjacent vertices of 4-cycle other than itself. Therefore, the subgroup $H=G_K$ is strongly quasiconvex by Theorem 1.11 in \cite{T2017}. We note that infinite index in $G_\Gamma$. Also $H$ is not hyperbolic and therefore $H$ is not stable.
\end{exmp}

\begin{rem}
The existence of the such subgroup $H \le G_\Gamma$ in Example~\ref{hayhayqua} implies that the group $G_\Gamma$ is not commensurable to any right-angled Artin group because all strongly quasiconvex subgroups of infinite index of a one-ended right-angled Artin group are free. We note that $G_\Gamma$ is not even quasi-isometric to any right-angled Artin group by the very recent work of Russell-Spriano-Tran (see Example 7.7 in \cite{RST}).
\end{rem}

\bibliographystyle{alpha}
\bibliography{hoanghung}
\end{document}